\documentclass[11pt,leqno]{amsart}
\usepackage{amsmath,amsfonts,amsthm,amscd,amssymb,graphicx}
\usepackage{epstopdf}
\numberwithin{equation}{section}
\usepackage{fullpage}
\usepackage{color}

\newcommand{\beq}{\begin{equation}}
\newcommand{\eeq}{\end{equation}}

\newcommand{\RR}{\ensuremath{\mathbb{R}}}

\newcommand{\bpm}{\begin{pmatrix}}
\newcommand{\epm}{\end{pmatrix}}

\newcommand{\CalB}{\mathcal{B}}

\newcommand{\CalX}{\mathcal{X}}
\newcommand{\Cal}{\mathcal}

\def\eps{\varepsilon }
\def\e{\varepsilon}

\newcommand\R{\mathbb R}

\newcommand\C{\mathbb C}
\newcommand\Z{\mathbb Z}
\newcommand\Zzer{\mathbb Z\setminus\{0\}}

\def\eps{\varepsilon}
\def\e{\varepsilon}

\def\ta{\tilde{a}}
\def\tmu{\tilde{\mu}}

\newcommand\br{\begin{remark}}
\newcommand\er{\end{remark}}
\newcommand\bp{\begin{pmatrix}}
\newcommand\ep{\end{pmatrix}}
\newcommand{\be}{\begin{equation}}
\newcommand{\ee}{\end{equation}}
\newcommand\ba{\begin{equation}\begin{aligned}}
\newcommand\ea{\end{aligned}\end{equation}}

\newcommand{\bap}{\begin{app}}
\newcommand{\eap}{\end{app}}
\newcommand{\begs}{\begin{exams}}
\newcommand{\eegs}{\end{exams}}
\newcommand{\beg}{\begin{example}}
\newcommand{\eeg}{\end{exaplem}}
\newcommand{\bpr}{\begin{proposition}}
\newcommand{\epr}{\end{proposition}}
\newcommand{\bt}{\begin{theorem}}
\newcommand{\et}{\end{theorem}}
\newcommand{\bc}{\begin{corollary}}
\newcommand{\ec}{\end{corollary}}
\newcommand{\bl}{\begin{lemma}}
\newcommand{\el}{\end{lemma}}
\newcommand{\bd}{\begin{definition}}
\newcommand{\ed}{\end{definition}}
\newcommand{\brs}{\begin{remarks}}
\newcommand{\ers}{\end{remarks}}

\newcommand{\CalE}{\mathcal{E}}
\newcommand{\CalF}{\mathcal{F}}

\newcommand{\CalN}{\mathcal{N}}

\newcommand{\ZZ}{{\mathbb Z}}

\newcommand{\CC}{{\mathbb C}}
\newcommand{\BbbA}{{\mathbb A}}

\newcommand{\BB}{{\mathbb B}}

\newcommand{\const}{\text{\rm constant}}
\newcommand{\Id}{{\rm Id }}
\newcommand{\Range}{{\rm Range }}

\newcommand{\sgn}{\text{\rm sgn}}
\newcommand{\rmi}{{\mathrm{i}}}
\newtheorem{theorem}{Theorem}[section]
\newtheorem{proposition}[theorem]{Proposition}
\newtheorem{corollary}[theorem]{Corollary}
\newtheorem{lemma}[theorem]{Lemma}

\theoremstyle{remark}
\newtheorem{remark}[theorem]{Remark}
\theoremstyle{definition}
\newtheorem{definition}[theorem]{Definition}

\newtheorem{example}[theorem]{Example}

\newtheorem{hypothesis}{Hypothesis}

\newcommand\tL{{L^\Pi}}
\newcommand\pL{{L^{\mathrm{pos}}}}
\newcommand\nL{{L^{\mathrm{neg}}}}
\newcommand\nLk{{L_k^{\mathrm{neg}}}}
\newcommand\nLze{{L_0^{\mathrm{neg}}}}
\newcommand\nZ{{\Z^{\mathrm{neg}}}}

\newcommand{\dom}{\text{\rm{dom}}}

\newcommand{\bb}{\mathbf{b}}
\newcommand{\bu}{\mathbf{u}}
\newcommand{\bv}{\mathbf{v}}

\newcommand{\bz}{\mathbf{z}}
\newcommand{\bF}{\mathbf{F}}

\newcommand{\tbv}{\tilde{\mathbf{v}}}
\newcommand{\bwe}{\mathbf{w}^\varepsilon}
\newcommand{\tbwe}{\tilde{\mathbf{w}}^\varepsilon}
\newcommand{\tS}{\tilde{\Sigma}}
\newcommand{\ale}{\alpha(\varepsilon)}
\newcommand{\oA}{\overline{A}}
\newcommand{\oB}{\overline{B}}

\renewcommand{\Re}{\mathrm{Re}}
\renewcommand{\Im}{\mathrm{Im}}

\title{
O(2) Hopf bifurcation of viscous\\
shock waves in a channel
}

\author{ Alin Pogan}
\address{ Indiana University, Bloomington, IN 47405}
\email{apogan@indiana.edu}
\thanks{Research of A.P. was partially supported under
NSF grant no. DMS-0300487}
\author{Jinghua Yao}
\address{ Indiana University, Bloomington, IN 47405}
\email{yaoj@indiana.edu}
\thanks{Research of J.Y. was partially supported under
NSF grant no. DMS-0300487}
\author{Kevin Zumbrun}
\address{Indiana University, Bloomington, IN 47405}
\email{kzumbrun@indiana.edu}
\thanks{Research of K.Z. was partially supported
under NSF grant no. DMS-0300487}
\begin{document}

\begin{abstract}
Extending work of Texier and Zumbrun in the semilinear non-reflection
symmetric case, we study $O(2)$ transverse Hopf bifurcation,
or ``cellular instability,''
of viscous shock waves in a channel,
for a class of quasilinear hyperbolic--parabolic systems
including the equations of thermoviscoelasticity.
The main difficulties are to (i) obtain Fr\'echet differentiability
of the time-$T$ solution operator by appropriate hyperbolic--parabolic
energy estimates, and (ii) handle $O(2)$ symmetry in the absence
of either center manifold reduction (due to lack of spectral gap)
or (due to nonstandard quasilinear hyperbolic-parabolic form)
the requisite framework for treatment by spatial dynamics
on the space of time-periodic functions, the two standard treatments
for this problem.
The latter issue is resolved by Lyapunov--Schmidt reduction of the time-$T$
map, yielding a four-dimensional problem with $O(2)$ plus approximate
$S^1$ symmetry, which we treat ``by hand''
using direct Implicit Function Theorem arguments.
The former is treated by balancing information obtained in
Lagrangian coordinates with that from an
augmented system.
Interestingly, this argument does not apply to gas dynamics
or magnetohydrodynamics (MHD), due to the infinite-dimensional family of Lagrangian symmetries
corresponding to invariance under arbitrary volume-preserving diffeomorphisms.
\end{abstract}

\date{\today}
\maketitle


\section{Introduction}\label{s:intro}

In this paper, we treat transverse Hopf bifurcation,
or ``cellular instability,'' of planar viscous shock waves in
an infinite channel with periodic boundary conditions,
for a class of hyperbolic--parabolic systems
including the equations of thermoviscoelasticity.
Transverse Hopf bifurcation has been treated in \cite{TZ2} for
semilinear equations.  The main differences here are
{\it partial parabolicity}/lack of parabolic smoothing
and {\it reflection symmetry} of the physical equations.
The former adds considerable technical difficulty to do with
the basic issue of regularity of the time-$T$ solution map,
as discussed for the 1D case in \cite{TZ3}.
The latter implies that the underlying bifurcation is not of planar
Hopf type, but, rather, a four-dimensional $O(2)$ Hopf bifurcation
as discussed for example in \cite{GSS}: roughly speaking, a ``doubled''
Hopf bifurcation coupled by nonlinear terms.

$O(2)$ Hopf bifurcation is typically treated by center manifold reduction followed
by transformation to a (doubly) angle
invariant normal form, and thereby
to a planar stationary bifurcation with $D_4$ symmetry in the (two) radial
coordinates.
Here, however, the linearized operator about the wave has no
spectral gap, hence standard center manifold theorems do not apply;
indeed, existence of a center manifold is unclear.
Instead, we proceed by the Lyapunov--Schmidt reduction
framework of \cite{TZ2},
applied to the time-$T$ evolution map of the underlying perturbation equations,
resulting in a $4$-dimensional stationary bifurcation problem with
$O(2)$ symmetry plus an additional ``approximate $S^1$ symmetry'' induced by
the underlying rotational linearized flow.
The latter is then analyzed ``by hand'', using direct
rescaling/Implicit Function Theorem arguments.

We note that, though there exist other methods suitable to treat related
problems without spectral gap, notably the spatial dynamics approach used by
Iooss, Sandstede--Scheel, and others- see in particular \cite{GSS,SS})-
the ``reverse temporal dynamics'' approach of \cite{TZ3} is the only one that has so far been successfully
applied (in the 1-D case) to the viscous shock solutions of physical,
partially parabolic systems.
Indeed, the main advantage of this method is
that it typically applies whenever there is an existing
time-evolutionary stability theory for the background equilibrium solution, which
in this case has already been developed.
A disadvantage of the method is, simply, that it is not the standard one, and so, as
turns out to be the case here, one
cannot always appeal to existing theory to treat the resulting reduced system.
Though elementary, the treatment of the nonstandard
finite-dimensional reduced system is thus a significant part of our analysis.

Regarding regularity, a critical aspect, as in the 1D
case \cite{TZ3}, is to work in Lagrangian rather than Eulerian coordinates,
in which hyperbolic transport modes become constant-coefficient linear rather
than quasilinear as in the Eulerian case, and certain key
variational energy estimates do not lose derivatives.
For our argument, we require also other favorable properties
related to stability of constant solutions that are evident in the Eulerian
formulation by existence of a convex entropy, but for multi-D
are in the Lagrangian formulation are less clear.
Fortunately, this issue has been addressed by Dafermos
using the ideas of involution and contingent entropy \cite{Da}
in a way suitable for our needs,
in particular yielding the necessary properties for the equations of
thermoviscoelasticity.

Lest one conclude that Eulerian and Lagrangian formulations
share identical properties, we point out
that for {\it gas dynamics and MHD, this is far from the case.}
Since the stress tensor in these cases
depends on the strain tensor through density alone, that is, only through
the Jacobian of the displacement map, it follows that the Lagrangian equations
are {\it invariant under any volume-preserving diffeomorphism,}
an infinite-dimensional family of symmetries preventing asymptotic
stability of constant solutions.  Meanwhile, the Eulerian equations, possessing
a convex entropy, automatically do have the property of asymptotic stability.
This represents a genuine difference between Eulerian and Lagrangian
formulations for gas dynamics and MHD, and an obstruction to the methods of
this paper.
%
We discuss in Section \ref{s:gas} various ideas how this might be overcome.

\medskip

{\bf Notation}: In what follows  $\bu:\RR\times [-\pi,\pi]\to\R^n$ is a smooth function, periodic on $[-\pi,\pi]$. Unless otherwise indicated, indices $j,k$ are in the
range $\{1, 2\}$, and summations in $j,k$ are
from 1 to 2. We denote
$\partial_j:=\frac{\partial}{\partial x_j}$ and
$\partial_t:=\frac{\partial}{\partial t}$,
$\alpha\in \mathbb{N}^2$,
and $D^{\alpha}:=\partial_x^\alpha$.

\subsection{Equations and assumptions}\label{assumptions}
Consider a one-parameter family of standing viscous planar shock
solutions \be\label{prof} \bu(x,t)=\bar \bu^\eps(x_1), \qquad
\lim_{x_1\to \pm \infty} \bar \bu^\eps(x_1)=\bu_\pm^\eps \quad \hbox{\rm
(constant for fixed $\eps$)}, \ee of a smoothly-varying family of
conservation laws
\begin{equation} \label{sys} \bu_t =\CalF(\eps; \bu):= \sum_{jk}\partial_j(B^{jk}(\eps;
\bu)\partial_k \bu)- \sum_j \partial_j F^j(\eps; \bu), \quad \bu\in \RR^n
\end{equation}\label{eq:2ncs}
on the periodic channel $x_1\in \R$, $x_2\in  [-\pi,\pi]_{per}$
possibly also subject to constraints
\be\label{constraints}
\sum_{j} M_j \partial_j u=0, \qquad M_j=\const \in \R^{r\times n}
\ee
preserved by the flow of \eqref{sys},
with associated linearized operators
\be\label{def-Leps}
L(\varepsilon)\bv:=\sum_{jk}\partial_j\Big(B^{jk}(\eps;\bar
\bu^{\varepsilon})\partial_k \bv-\partial_\bu B^{jk}(\eps;\bar
\bu^{\varepsilon})\bv\partial_k\bar
\bu^{\varepsilon}\Big)-\sum_j\partial_j (\partial_\bu F^j(\eps;\bar
\bu^{\varepsilon})\bv).
\ee
Typically, the bifurcation parameter $\varepsilon$ measures
shock amplitude or other physical parameters.
Here, the linear operator $L(\eps)$ is considered as a closed linear operator on $L^2(\R\times [-\pi,\pi],\C^n)$ with domain $\dom(L(\eps))=H^2(\R\times[-\pi,\pi],\C^n)$, and
the functions $B^{jk}:(-\delta,\delta)\times\R^n\to\R^{n\times n}$ and $F^j:(-\delta,\delta)\times\R^n\to\R^n$ are smooth in $\bu$, see also
Hypothesis (H0) below.
Equations (\ref{sys}) are typically shifts  $B^{jk}(\eps;
\bu)=B^{jk}(\bu)$, $F^1(\eps; \bu):= f^1(\bu)-s(\eps) \bu$ of a single
equation
$$ \bu_t = \sum_{jk}\partial_j(B^{jk}(\bu)\partial_k \bu) - \sum_{j}\partial_j F^j(\bu)
$$
written in coordinates $\tilde x =( x_1-s(\eps)t, x_2)$
moving with traveling-wave solutions $\bu(x,t)=\bar
\bu^\eps(x_1-s(\eps)t)$ of varying
speeds $s(\eps)$.
Profiles $\bar \bu^\eps$ satisfy the standing-wave ODE \be\label{ode-u-eps}
B^{11}(\eps;\bu)\bu'=F^1(\eps;\bu)- F^1(\eps; \bu^\eps_-). \ee

We assume, further, that there are {\it augmented variables}
$z:=(\bu, z(\bu))\in \RR^N$, $N\geq n$, satisfying an enlarged system of
conservation laws whenever $\bu$ satisfies \eqref{sys}--\eqref{constraints},
and a further invertible coordinate
change $z\to w$ yielding the ``partially symmetric
hyperbolic--parabolic form'' \cite{Z2}
\be\label{aug}
\BbbA^0(\eps;w) w_t =
\sum_{jk}\partial_j(\BB^{jk}(\eps;
w)\partial_k w)- \sum_j
\BbbA^j(\eps; w) \partial_jw
+ \begin{pmatrix} 0\\ \tilde g \end{pmatrix},
\;
\hbox{\rm with $\tilde g={\mathcal{O}}(|w_x|^2)$.}
\ee
%
We shall use the notation
\begin{equation}\label{Aj-eps}
A^j:(-\delta,\delta)\times\R^n\to\R^{n\times n},\quad A^j(\eps;\bu)=F_\bu^j(\eps;\bu);
\end{equation}
\[
A_\pm^j,B_\pm^{jk}:(-\delta,\delta)\to\R^{n\times n}
\]
\begin{equation}\label{ABj-eps-pm}
A_\pm^j(\eps)=F_\bu^j(\eps;\bu_\pm^\eps)=A^j(\eps;\bu_\pm^\eps),\quad B_\pm^{jk}(\eps)=B^{jk}(\eps;\bu_\pm^\eps).
\end{equation}
\[
\BbbA_\pm^j,\BB_\pm^{jk}:(-\delta,\delta)\to\R^{n\times n}
\]
\begin{equation}\label{bABj-eps-pm}
\BbbA_\pm^j(\eps)= \BbbA^j(\eps;z_\pm^\eps),\quad \BB_\pm^{jk}(\eps)=\BB^{jk}(\eps;z_\pm^\eps).
\end{equation}
In what follows, if $A$ is an $n\times n$ matrix we will use lower subscripts for the block decomposition
$A=\left(\begin{array}{cc} A_{11} & A_{12} \\ A_{21} & A_{22}
\end{array}\right)$, where $A_{11}$ is an $(n-r)\times(n-r)$ matrix and $A_{22}$ is an $r\times r$ matrix.
If $\BbbA$ is $N\times N$, we will use the same notation for
block decompositions in $N-r$ and $r$ dimensional blocks.

\subsubsection{Structural conditions}
We make the following structural assumptions:
\begin{enumerate}

\item[(A1)] For every $j,k\in\{1,2\}$ there exists functions $b^{jk}:(-\delta,\delta)\times \R^n\to\R^{r\times r}$ and function $\tilde{b}^{jk}:(-\delta,\delta)\times \R^N\to\R^{r\times r}$  such that $B^{jk}$ and
$\BB^{jk}$  have the representations
\be\label{rep-Bjk}
B^{jk}(\eps;\bu)=\left(\begin{array}{cc} 0 & 0 \\ 0 & b^{jk}(\eps;\bu),
\end{array}\right),
\qquad
\BB^{jk}(\eps;w)=\left(\begin{array}{cc} 0 & 0 \\ 0 & \tilde{b}^{jk}(\eps;w),
\end{array}\right).
\ee
Moreover, the first $(n-r)$ components of $F^j(\eps, \bu)$, $j\in\{1,2\}$ are
{\it linear in} $\bu$.


\item[(A2)] There exists a matrix-valued function $A^0:(-\delta,\delta)\times\R^n\to\R^{n\times n}$
smooth in $\bu\in\R^n$,  positive definite and having a block-diagonal structure such that
\begin{enumerate}
\item[(i)] $A^0_{11}(\eps;\bu) A^j_{11}(\eps;\bu)$ is symmetric for any $\eps\in(-\delta,\delta)$, $\bu\in\R^n$;
\item[(ii)] $A^0_{11}(\eps;\bu)A^1_{11}(\eps;\bu)$ is either positive definite or negative definite for any $\eps\in(-\delta,\delta)$, $\bu\in\R^n$;
\item[(iii)] there exists a constant $\theta>0$ such that
\[
\sum_{j,k}\bv_j\cdot
\Big(A^0_{22}(\eps;\bu)b^{jk}(\eps;\bu) \bv_k\Big)\geq \theta \sum_j |\bv_j|^2\quad\mbox{for all}\quad \bv_1, \bv_2\in\mathbb{R}^r.
\]
\end{enumerate}

\item[(A3)]
$\BbbA^0$ is block-diagonal, symmetric positive definite and:
\begin{enumerate}
\item[(i)] $\BbbA^j_{11}(\eps;w)$ is symmetric for any $\eps\in(-\delta,\delta)$, $w\in\R^N$;
\item[(ii)] $\BbbA^1_{11}(\eps;w)$ is either positive definite or negative definite for any $\eps\in(-\delta,\delta)$, $w\in\R^N$;
\item[(iii)] $\BbbA_\pm^j(\eps)$ is symmetric for any $\eps\in(-\delta,\delta)$;
\item[(iv)] there exists a constant $\theta>0$ such that
\[
\sum_{j,k}\bv_j\cdot
\Big(\BbbA^0_{22}(\eps;w)\tilde b^{jk}(\eps;w) \bv_k\Big)\geq \theta \sum_j |\bv_j|^2\quad\mbox{for all}\quad \bv_1, \bv_2\in\mathbb{R}^r.
\]
\end{enumerate}
\end{enumerate}

\noindent
To (A1)--(A3), we add the following more detailed hypotheses. Here
and elsewhere, $\sigma(M)$ denotes the spectrum of a matrix or
linear operator $M$.
\begin{enumerate}
\item[(H0)] For any $j,k\in\{1,2\}$ the functions  $F^j$,$B^{jk}$, $\BbbA^j$,
$\BB^{j,k}$ are of class $\mathcal{C}^{\nu}$, for some $\nu\ge 5$.

\item[(H1)]  For any $x_1\in\RR$, $0$ is not an eigenvalue of $\BbbA_{11}^1(\eps;\bar \bu^\eps(x_1))]_{11}$. Moreover, each eigenvalue of $\BbbA_{11}^1(\eps;\bar \bu^\eps(x_1))]_{11}$   has a multiplicity independent of $x_1\in\RR$.

\item[(H2)] $\sigma(A_\pm^1(\eps))$ is real, semisimple, and nonzero,

\item[(H3)]
There is no eigenvector of $\sum_j \BbbA^j_\pm(\eps) \xi_j$ lying in
$\ker \sum_{j,k}\BB^{jk}_\pm(\eps) \xi_j \xi_k$, for any $\xi\in \RR^2$
and all $\e\in (-\delta,\delta)$
(Kawashima's {\it genuine coupling condition} \cite{Kaw}).

\item[(H4)]  Considered as connecting orbits of \eqref{ode-u-eps}, $\bar \bu^\eps$
lie in an $\ell$-dimensional manifold, $\ell\ge 1$, of solutions
(\ref{prof}), obtained as a transversal intersection of the unstable
manifold at $\bu^\eps_-$ and the stable manifold at $\bu_+^\eps$.
(In the most typical case of a {\it Lax-type} shock \cite{L},
$\ell=1$ and the manifold of solutions
consists simply of the set of $x_1$-translates of a single wave.)
\end{enumerate}

\noindent
Finally, we make the key assumption of {\it reflection symmetry}:
\begin{enumerate}
\item[(B1)] Equations \eqref{sys} are invariant under
$S:x_2\to -x_2$, $\bu\to M\bu$, for $M\in \R^{n\times n}$ constant.
\end{enumerate}

\begin{remark}\label{entrmk}
Conditions (A1)--(A2) are analogous to the 1D conditions of \cite{TZ3},
used to obtain Fr\'echet differentiability of the nonlinear source term
in the time-$T$ solution map for the perturbation equations of \eqref{sys}
about a standing shock.
Condition (A3) is analogous to the multi-D conditions of \cite{Z1,Z2},
used to obtain damping-type energy estimates and high-frequency resolvent bounds.
To obtain a system of form \eqref{aug} satisfying (A3), it is
sufficient that (a) {\it the system of conservation laws
in the augmented variable $z$ inherit the structural ``well-posedness''
conditions
(A3)(i),(ii),(iv)} (in practice no issue) and (b) the system of conservation laws in $z$ {\it possess
a convex entropy $\eta(z)$ in a neighborhood of endstates $z_\pm^\eps$};
see the discussion surrounding Eg. 2.15 and Section 2.1 of \cite{Z2}.
\er

\br
With slight further effort,
we may replace as in \cite{Z1,Z2} the uniform ellipticity conditions
(A2)(iii) and (A3)(iv) with the spectral conditions
$
\sigma \Big(
\sum_{j,k} A^0_{22}(\eps;\bu)b^{jk}(\eps;\bu) \xi_j
\xi_k\Big)\geq \theta |\xi|^2$,
for all $\xi\in \R^2$,
using G\"arding's inequality instead of direct integration by parts
in the energy estimates of Section \ref{s:energy}.
See the proof of \cite[Proposition 5.9]{Z1} or the proof of \cite[Proposition 1.16]{Z2}.
However, this is not needed for the physical applications from continuum
mechanics that we have in mind.
\er

\br\label{O2rmk}
Condition (B1) together with translation-invariance in $x_2$, implies
O(2) symmetry in the perturbation equations around the (symmetric, since
constant in $x_2$-direction) background shock solutions, with
$R(\theta): \bu(x_1, \cdot, t)\to \bu(x_1, (\cdot +\theta)\mathrm{mod}\;2\pi, t)$
corresponding to rotation and $S$ reflection, $R(\theta)S=SR(-\theta)$.
That is, ``rotation'' in this context should be thought of as $x_2$-translation.
\er

\begin{example}\label{viscoeg}
The equations of isothermal viscoelasticity in Lagrangian
coordinates are
\be\label{thermovisc}
 \xi_{tt} - \nabla_X\cdot \big(DW(\nabla\xi) + \mathcal{Z}(\nabla\xi,
\nabla\xi_t)\big)=0,
\ee
where $\xi$ denotes deformation of an initial reference configuration
of constant temperature and density, and $\nabla_X\cdot$
denotes divergence, taken row-wise across a matrix field.
The elastic potential $W$ is a function of the deformation
gradient $F:=\nabla \xi$ and the viscous stress tensor $\mathcal{Z}$
a function of $F$ and $F_t$ obeying the Claussius-Duhem inequality
$\mathcal{Z}(F,Q):Q\geq 0$, where ``$:$''
denotes Frobenius matrix inner product;
see \cite{A,Da,BLeZ}.
Expressing these equations as a second-order system in $F$ and $u:=\xi_t$,
we obtain a system of first-order linear hyperbolic equations
$F_t-\nabla u=0$ in $F$ coupled with second-order parabolic equations
$u_t -  \nabla_X\cdot \big(DW(F) + \mathcal{Z}(F,\nabla_X u)\big)=0$
in $u$.
In the strict case $Q:\mathcal{Z}(F,Q)\gtrsim |Q|^2$,
the latter satisfies the uniform ellipticity
condition (A2)(iii), and we obtain (A1)--(A2), with constraints
$\nabla_X \times  F=0$, and $A^0=\Id$.

The delicate point is to verify (A3), and in particular (A3)(iii)
for an appropriate augmented system.
But, this follows whenever $W$ is polyconvex in a neighborhood of
$z_\pm^\eps$, i.e., a
strictly convex function of $F$, its determinant $\det F$, and its adjugate,
or transposed matrix of minors $F^\sharp$, as shown by Dafermos \cite{Da} using
the method of {\it contingent entropies} and considering the system
in terms of the extended variable
$z:=(F,u,F^\sharp, \det F)$, by the fact that $W(z)$ is then automatically
a (``viscosity-compatible'') strictly convex
entropy for the parabolic flow, from which the result follows as
described in Remark~\ref{entrmk}, by a construction similar to that of \cite[Eg. 1.25]{Z2}, .
\end{example}

\begin{example}\label{gaseg}
Similar considerations as in Example \ref{viscoeg}
 yield that the equations of thermoviscoelasticity satisfy our assumptions,
so long as the thermoelastic potential $e=e(F,S)$
is a strictly convex function of $(F,F^\sharp, \det F)$ and entropy $S$.
However, for ordinary gas dynamics, $e$
is a convex function of $\tau:=\det F$ and $S$ alone, where $\tau$ denotes
specific volume.  Thus, considered as a function of $(F,F^\sharp,\det F,S)$,
it is nonstrictly convex, and indeed it is readily verified that (H4) fails
(see \S \ref{s:gas}).
\end{example}

\br\label{tedrmk}
Though the construction of Example \ref{viscoeg} and the
conditions (A1)--(A3) in terms of two different coordinatizations
may seem overly complicated, we do not see a way to shorten this
description.  In particular, note that for the equations of isentropic
viscoelasticity, condition (A3)(iii) typically fails
for coordinates $\bu$, in which \eqref{sys} does not possess a convex entropy,
hence (by results of \cite{KaS})
the linearized equations are not symmetrizable near
endstates $\bu_\pm^\eps$.
On the other hand, condition (A1) fails in the coordinates $z$,
since the first $N-r$ coordinates, now including nonlinear functions of
the first $n-r$ coordinates of $\bu$, are no longer linear.
As our arguments require both of these properties, we therefore seem to require
both coordinatizations as well.
\er

\begin{example}
Eulerian gas dynamics or MHD with artificial viscosity, i.e., strictly parabolic
second-order terms, satisfy conditions (A1)--(A3), (H0)--(H4) automatically
with the single coordinate $z(\bu)=\bu$.
This admits a much simpler treatment of regularity, as, e.g., in \cite[\S 6]{TZ2}.
\end{example}

\subsection{O(2) Hopf bifurcation}\label{mot}
Before stating our results, we recall the standard $O(2)$
Hopf bifurcation scenario in finite dimensions, following
Crawford and Knobloch \cite{CK}.
After Center manifold/Normal form reduction, this takes the form, to cubic order, of
\ba\label{normz}
\dot z_1&= (\varepsilon \hat\varkappa(\varepsilon) +\hat\chi(\varepsilon) \rmi) z_1 + (\hat \Lambda |z_1|^2 + \hat \Gamma |z_2|^2)z_1,\\
\dot z_2&= (\varepsilon \hat\varkappa(\varepsilon)+\hat\chi(\varepsilon) \rmi) z_2 + (\hat \Lambda |z_2|^2 + \hat \Gamma |z_1|^2)z_2,\\
\ea
where $\varepsilon\in \RR$ is a bifurcation parameter, $\varkappa,\chi\in \RR$
are nonzero bifurcation coefficients,
and $z_j, \hat \Lambda, \hat \Gamma \in \C$.
Model \eqref{normz} has $O(2)$ symmetry group consisting of rotation
$R(\theta): (z_1,z_2)\to (z_1e^{\rmi\theta},z_2e^{-\rmi\theta})$
and reflection $S: (z_1,z_2)\to (z_2,z_1)$, with $SR(\theta)=R(-\theta)S$,
and also an additional $S^1$ symmetry $T(\beta): (z_1,z_2)\to (z_1e^{\rmi\beta},z_2e^{\rmi\beta})$ associated with normal form.
The linearization about the trivial equilibrium solution $(z_1,z_2)=(0,0)$
features a pair of double eigenvalues
$$
\lambda_\pm(\varepsilon)= \hat\varkappa(\varepsilon) \eps \pm \rmi\hat\chi(\varepsilon)
$$
crossing the imaginary axis at $\varepsilon=0$: an {\it equivariant Hopf bifurcation} with double multiplicity forced by reflection symmetry.
Noting that radial equations decouple from angular equations as
\ba\label{normr}
\dot r_1&= \varepsilon \hat\varkappa(\varepsilon) r_1 + (
\Re \hat \Lambda |r_1|^2 + \Re \hat \Gamma |r_2|^2)r_1,\\
\dot r_2&= \varepsilon \hat\varkappa(\varepsilon) r_2 + (\Re \hat \Lambda |r_2|^2 + \Re \hat \Gamma |r_1|^2)r_2,\\
\ea
$r_j:=|z_j|$, we find that periodic solutions are exactly equilibria for the
planar radial system \eqref{normr}.
Under the nondegeneracy conditions
\be\label{nondeg}
\Re \hat \Lambda|_{\varepsilon=0}\neq 0,
\quad
\Re(\hat \Lambda +\hat \Gamma)|_{\varepsilon=0}\neq 0,
\quad
\Re (\hat \Lambda-\hat \Gamma)|_{\varepsilon=0}\neq 0,
\ee
it is readily seen that the periodic solutions consist, besides
the trivial solution $(r_1,r_2)=(0,0)$ exactly of
``traveling'' (or ``rotating'') wave solutions $(r_1,r_2)\equiv (r_*,0)$ or $(0,r_*)$ and ``standing'' (or ``symmetric'') wave solutions
$(r_1,r_2)\equiv (r_*,r_*)$
consisting of a nonlinear superposition of counter-rotating traveling waves,
$r_*\neq 0$, with associated radial bifurcations of pitchfork type
$|r|\sim \sqrt{\varepsilon}$.

Restricting attention to periodic solutions with period $T(\varepsilon)$
near the linear period $T_*(\varepsilon):= 2\pi/\hat \chi(\varepsilon)$,
and noting that spurious radial equilibria introduced by $\Re(\hat\Lambda-\hat\Gamma)=0$
will have different periods in $z_1$, $z_2$ unless $\Im(\hat\Lambda-\hat\Gamma)=0$
as well, we find that \eqref{nondeg} may be weakened to
\be\label{wnondeg}
\Re \hat \Lambda|_{\varepsilon=0}\neq 0,
\quad
\Re(\hat \Lambda +\hat \Gamma)|_{\varepsilon=0}\neq 0,
\quad
(\hat \Lambda-\hat \Gamma)|_{\varepsilon=0}\neq 0.
\ee

\subsubsection{Alternative treatment via the displacement map}\label{dtreat}
Alternatively, we show in \S \ref{Proof-MT} that $O(2)$ Hopf bifurcation
after Lyapunov--Schmidt reduction of the time-$T$ displacement map $F_j:=
z_j(T)-z_j(0)$ for $z_j(0):=a_j$ takes the form, to cubic order, of
a {\it two-parameter stationary bifurcation}:
\ba\label{dispa}
0=F_1(a_1,a_2,\varepsilon ,\mu)&= (\varepsilon \varkappa(\varepsilon,\mu)
 +\chi(\varepsilon) \mu \rmi) a_1 + (\Lambda |a_1|^2 + \Gamma |a_2|^2)a_1,\\
0=F_2(a_1,a_2,\varepsilon ,\mu)&= (\varepsilon \varkappa(\varepsilon,\mu)+\chi(\varepsilon) \mu \rmi) a_2 + (\Lambda |a_2|^2 + \Gamma |a_1|^2)a_2\\
\ea
in four dimensions,
where $a_j,\Lambda, \Gamma\in \CC$, $\varkappa, \chi\in \R$ are
nonzero bifurcation coefficients, and $ \varepsilon, \mu \in \R$
are bifurcation parameters, with $\mu $ measuring the difference between
$T$ and the linear period $T_*(\varepsilon)$.
It is readily checked that, again, zeros of \eqref{dispa}, corresponding
to periodic solutions for the original problem, consist, besides the
trivial solution $(a_1,a_2)=(0,0)$, exactly of
traveling waves $(a_1,a_2)=(a_*,0)$ or $(0,a_*)$ and standing waves
$(a_1,a_2)=(a_*,e^{\rmi\theta} a_*)$, $a_*\neq 0$, $\theta\in \R$,
each of pitchfork type $|a|\sim \sqrt{\varepsilon}$,
under the nondegeneracy conditions
\be\label{nohatwnondeg}
\Re \Lambda|_{\varepsilon=0}\neq 0,
\quad
\Re(\Lambda +\Gamma)|_{\varepsilon=0}\neq 0,
\quad
(\Lambda-\Gamma)|_{\varepsilon=0}\neq 0.
\ee
This gives a different, more direct (though higher-dimensional),
route to $O(2)$ Hopf bifurcation avoiding Center Manifold or Normal form reductions, the simple form of the truncated system \eqref{dispa}
being forced rather by symmetry/time-averaging.
The extension to the full
system then proceeds by rescaling/Implicit Function Theorem arguments,
as described in \S \ref{Proof-MT} and Appendix \ref{sec:jacobian}.

\subsection{Statement of the main result}\label{mainresult}
We are now ready to describe our main results.
Note first that, by independence of coefficients of $L(\eps)$ on
the $x_2$-coordinate, the spectra of $L(\eps)$ may be decomposed
into spectra associated with invariant subspaces of functions $e^{\rmi kx_2}f(x_1)$
given by Fourier decomposition, on which $L(\eps)$ acts as
an ordinary differential operator $L_k(\eps)$ in $x_1$.
The operator $L_0(\eps)$ is exactly the linearized operator for the associated
one-dimensional problem, while the operators $L_k(\eps)$, $k\in\ZZ\setminus\{0\}$, govern the evolution
of ``transverse modes'' with Fourier wave number $k$.

We shall assume {\it one-dimensional spectral stability},
or stability of $L_0(\eps)$, in the sense of \cite{Z1,Z2}, which
is typically expressed
in terms of an {\it Evans function} associated with the wave.
This expresses in a generalized sense that the multiplicity-$\ell$ zero-eigenvalues guaranteed by (H4), which are also embedded in the essential spectrum
of $L_0(\eps)$, are the only spectra of $L_0(\eps)$ contained in the nonstable complex
half-plane $\CC_+=\{\lambda\in\CC:\Re \lambda \geq 0\}$.

At the same time, we will assume that there exists a conjugate pair of eigenvalues
\be\label{scene}
\lambda_\pm(\eps)= \gamma(\eps)\pm \rmi\omega(\eps),
\quad \gamma'(0), \omega(0)\neq 0
\ee
crossing the imaginary axis as $\eps$ crosses the bifurcation value $\eps=0$,
associated with transverse Fourier modes $k=\pm k_*\neq 0$,
and that these are each of the minimal (in the presence of reflective
symmetry) multiplicity two.
By the assumed reflective symmetry in $x_2$,
$\lambda_+(\eps)$ and $\lambda_-(\eps)$ are thus associated to subspaces
with eigenbases
\be\label{bespace}
\hbox{\rm $e^{\pm \rmi k_* x_2}w^\eps(x_1)$  and
$e^{\mp \rmi k_* x_2}\overline{w^\eps}(x_1)$,}
\ee
respectively, where $\bar{ }$ denotes complex conjugate.
These hypotheses are gathered in \S \ref{Proof-MT} as condition $(D_\eps)$;
see \S \ref{Proof-MT} for further details.
Together, they comprise a {\it spectral transverse $O(2)$ Hopf bifurcation;}

\medskip

Define now the exponentially-weighted function space
\be\label{def-X1'}
X'_1=H_\eta^4(\R,\C^n):=H^4(\R,\C^n;e^{\eta(1+|x_1|^2)^{1/2}}dx_1),
\quad \eta <<1,
\ee
with its natural Hilbert space norm and scalar product.
Then, the main result of this paper is:
\begin{theorem}\label{main-result}
Let $\bar \bu^\eps$, $\eps\in(-\delta,\delta)$, be a one-parameter family of standing viscous planar shock solutions of \eqref{sys} satisfying
Hypotheses (A1)--(A3), (B1), (H0)--(H4) and $(D_\eps)$ (given in \S\ref{O2-bif}). Then:

(i) Existence of time-periodic solutions $\hat u^\eps$
close in $X_1'$ norm to $\bu^\eps$ with period $T$
close to the linearized
value $T_*(\eps)=2\pi/\omega(\eps)$
is equivalent to satisfaction of one of an $\ell$-parameter family
\eqref{new-system-bif} of equations of
form \eqref{dispa} plus higher-order perturbations,
indexed by $\bb\in \R^\ell$ relating $\eps$,
a parameter $\mu$ measuring the difference between
$T$ and $T_*$, and the
projection $\Pi^\eps (\hat u^\eps -\bu^\eps)|_{t=0}$,
appropriately coordinatized
as $(a_1,a_2)\in \CC^2$, where $\Pi^\eps $ is the total eigenprojection of
$L(\eps)$ onto the eigenspace \eqref{bespace} associated with
eigenvalues $\lambda_\pm (\eps)$.

(ii) Under the genericity
assumption \eqref{geneq}
(Hypothesis \ref{genericity}, \S \ref{Proof-MT})
analogous to \eqref{nohatwnondeg},
equation \eqref{sys} exhibits an $\ell$-fold
$O(2)$-Hopf bifurcation from $\bar\bu^\eps$, namely, the family of
$X_1'$-close (nontrivial) periodic solutions with nearby periods consists
precisely of $4$ smooth
families indexed by $\bb$ sufficiently small, of bifurcating solutions
$$
\sqrt{\eps}/C\leq \|\hat u^\eps-\bu^\eps\|_{X_1'}\leq C\sqrt{\eps}
$$
with $(a_1,a_2)$ $\eps$-close to each of the ``traveling-'' and
``standing-wave'' solutions
$(a_*(\bb),0),(0,a_*(\bb))$ and $(a_{\natural}(\bb),a_{\natural}(\bb))$
of the associated cubic truncated system, as described in \S \ref{mot}.
These are, variously,
of supercritical ($\eps>0$) or subcritical ($\eps<0$) type, depending on
model parameters $\Lambda(\bb)$, $\Gamma(\bb)$ at $\bb=0$.
In the Lax case $\ell=1$, the ``traveling-wave'' type solutions are
actual traveling waves $\overline{\bu}^\eps(x,t)= h^\eps(x_1,x_2-dt)$ with respect
to the transverse $x_2$ direction.
\end{theorem}

\begin{remark}\label{(D-epsilon-rem)}
It is interesting to note that the case $\omega(0)=0$ of a transverse
stationary bifurcation can be converted to the case $\omega(0)\neq 0$ of a
transverse Hopf bifurcation, and vice versa, by the introduction of a moving coordinate frame $x_2 \to x_2-dt$, inducing a shift $\lambda \to \lambda- \rmi kd$
in eigenvalues associated with Fourier number $k$.
Setting $d=\omega(0)/k_*$, this converts the scenario \eqref{scene} to that of an ordinary
(non-reflective symmetric) bifurcation involving a pair of roots crossing
at $\gamma(0)\pm 2\rmi\omega(0)$ plus a pair of roots crossing at $\lambda=0$,
i.e., a translationally-invariant stationary bifurcation in the moving coordinate frame.
This type of bifurcation has been treated recently in \cite{M} for the strictly parabolic semilinear case.
Likewise, a stationary $O(2)$ bifurcation involving a double eigenvalue
$\lambda(\eps)=\gamma(\eps)$ with wave numbers $k=\pm k_*$ can be converted
by the change of coordinates $x_2\to x_2-dt$ to an ordinary (non-$O(2)$-symmetric)
Hopf bifurcation $\lambda_\pm(\eps)=\gamma(\eps)\pm \rmi dk_*$, and treated as
in \cite{TZ2} to yield a time-periodic solution in the moving coordinate frame.
In the Lax case $\ell=1$, the uniqueness/shift invariance argument of
\S \ref{Proof-MT} yields the further information
that this is a traveling wave in $x_2$, as shown by direct (stationary)
argument in \cite{M}.
Thus, there is some overlap in the results obtainable by the methods here and those of \cite{M}; the difference in the $O(2)$ Hopf case is that we obtain full information on {\it all} time-periodic solutions and not only traveling waves.
As noted in \cite{M}, an example of the latter case
arises in MHD, as follows by observations of \cite{FT}.

A third, more degenerate,
possibility not yet treated is the case of a stationary $O(2)$ bifurcation for which the associated eigenfunctions $e^{\pm \rmi k_*x_2}
w(x_1)$ have a genuinely complex profile $w$ for which
$e^{\pm \rmi k_*x_2} \overline{w}(x_1)$ are linearly independent; in this case the multiplicity at $\lambda=0$ is four, and Lyapunov-Schmidt reduction of either the stationary problem or the time-$T$ evolution map for the associated Hopf bifurcation is a more general, $4$-dimensional bifurcation (mapping) problem in
$(z_1,z_2)$ with rotational but not reflectional symmetry.
This would be quite interesting to understand.
\end{remark}
\subsection{Discussion and open problems}
The Lyapunov-reduction/time-$T$ displacement map argument used here
serves as a substitute for the Center Manifold/Normal form reduction or
Lyapunov-reduction/spatial dynamics methods that have been used in other
contexts.
It is interesting to contrast the important use of additional $S^1$ symmetry in
these arguments, corresponding roughly to invariance with respect to
time-evolution.
This is imposed by force in normal form reduction, and appears naturally in
the spatial dynamics approach framed in the space of time-periodic solutions.
In our argument, we use the fact that time-evolution is an approximate
$S^1$ symmetry in a similar way, to restrict the possible forms arising at the
level of cubic approximation; see Remark~\ref{r4.6}.

An interesting open problem would be to carry out a similar analysis
for physical, partially parabolic, systems either in one- or multi-dimensions
using spatial dynamics techniques as done in \cite{SS} in the one-dimensional semilinear strictly parabolic case.
This appears to require both additional care in the choice of
spaces/analytical framework, and
additional theory to cope with absence of parabolic smoothing/compactness.
However, a possible advantage might be to remove the dependence on Lagrangian coordinates
that prevents for the moment the treatment of gas dynamics and MHD.
Some other ideas using the present framework are mentioned in Section \ref{s:gas}.

We note that the same issues obstructing multi-dimensional bifurcation
analysis obstruct also the proof of a multi-dimensional conditional
stability result similar to that obtained in \cite{Z5}
in the one-dimensional case- specifically, incompatibility
between Lagrangian form needed to obtain regularity needed for the
center-stable manifold reduction step and high-frequency resolvent estimates
needed for time-asymptotic decay estimates- making this issue one
of independent interest.

Though we do not carry it out here, spectral stability information on bifurcating solutions should be in principle available via the same reduced
displacement map; indeed, it should be ``reverse-engineerable''
from the standard normal-form analysis
via the relation (through time-integration) of
$\tilde \Lambda$, $\tilde \Gamma$ to $\Lambda$, $\Gamma$.
A very interesting open problem is to prove a full nonlinear stability result
under the assumption of spectral stability for a class of time-periodic
multi-dimensional solutions including the bifurcating time-periodic waves established here, similarly as was done in one dimension in \cite{BeSZ} in the strictly parabolic case.
Another interesting direction is the treatment of {\it spinning} shocks
and detonations in a cylindrical duct \cite{KS}, which should be treatable by similar arguments.

Finally, we note that the phenomenon of cellular instability/transverse $O(2)$ bifurcation of shock waves has so far been demonstrated mathematically only for a single example in MHD \cite{FT,M}.
The systematic cataloguing of this phenomenon for other waves and models by either numerical or analytic means we regard as an extremely interesting
open problem.

\vspace{0.3cm}

\noindent{\bf Acknowledgement.} Thanks to Arnd Scheel for helpful conversation
improving the exposition.

\section{Variational energy estimates}\label{s:energy}
In this section we introduce a useful energy functional associated
with  the perturbation equations for \eqref{sys},
and   prove the key energy estimate it satisfies; see Proposition~\ref{energyest} below. We start by
linearizing equation\eqref{sys} about $\bar\bu^\eps$. The linearized equation reads as follows:
\be\label{lin-sys}
\partial_t \bv=L(\varepsilon)\bv,
\ee
where the linear operators $L(\eps)$ are defined in \eqref{def-Leps}.
Next, we note that if $\bu$ is a solution of \eqref{sys} and if we denote by
$\bv(x,t):=\bu(x,t)-\bar \bu^{\eps}(x_1)$, then $\bv$ satisfies the perturbed equation:
\ba\label{pert-sys}
\partial_t \bv=
&L(\eps)\bv+\sum_{jk}\partial_j\Big(B^{jk}(\bv+\bar\bu^\eps)\partial_k(\bv+\bar\bu^\eps)-B^{jk}(\bar
\bu^{\varepsilon})\partial_k\bar \bu^{\eps}-B^{jk}(\bar
\bu^{\varepsilon})\partial_k \bv\\
&\quad-\partial_\bu B^{jk}(\bar
\bu^{\varepsilon})\bv\partial_k\bar \bu^{\varepsilon}\Big)-\sum_j\partial_j\Big(F^{j}(\bv+\bar\bu^\eps)-F^j (\bar \bu^{\eps})-\partial_\bu F^j
(\bar\bu^{\eps})\bv \Big).
\ea
Using assumptions $(A1)$, $(A2)$ and $(H0)$, we infer that
\be\label{part1-e-est}
\Big(F^{j}(\bv+\bar\bu^\eps)-F^j (\bar \bu^{\eps})-\partial_\bu F^j (\bar
\bu^{\eps})\bv \Big)=
\begin{pmatrix} 0\\ q^{j,\eps}(\bv)  \end{pmatrix};
\ee
\be\label{part2-e-est}
B^{jk}(\bv+\bar\bu^\eps)\partial_k(\bv+\bar\bu^\eps)-B^{jk}(\bar
\bu^{\varepsilon})\partial_k\bar \bu^{\eps}-B^{jk}(\bar
\bu^{\varepsilon})\partial_k \bv-\partial_\bu B^{jk}(\bar
\bu^{\varepsilon})\bv\partial_k\bar \bu^{\varepsilon}
=\begin{pmatrix} 0\\p^{j,\eps}(\bv)
\end{pmatrix}.
\ee
The functions $p^{jk,\eps},q^{j,\eps}$, $j,k=1,2$,  are defined as follows: for $\bv\in H^s(\R\times [-\pi,\pi],\C^n)$, $s\geq 2$,
\be\label{def-q-j}
q^{j,\eps}(\bv)=\int_0^1(1-t)\partial_\bu^2F_2^j(\bar\bu^\eps+t\bv)\bv\bv\,dt;
\ee
\be\label{def-p-jk}
p^{jk,\eps}(\bv)=\int_0^1\partial_\bu b^{jk}(\bar\bu^\eps+t\bv)\bv\partial_k\bv_2\,dt+
\int_0^1(1-t)\partial_\bu^2b^{jk}(\bar\bu^\eps+t\bv)\bv\bv\partial_k\bar\bu_2^\eps\,dt.
\ee
Using the functions introduced in \eqref{def-q-j} and \eqref{def-p-jk}, we obtain from \eqref{lin-sys}, \eqref{part1-e-est} and \eqref{part2-e-est}  that the perturbed equation \eqref{pert-sys} can be written as
\begin{equation}\label{inter1-pert}
\bv_t-L(\eps) \bv=\sum_{j,k}\partial_j \begin{pmatrix} 0\\ p^{jk,\eps}(\bv)\end{pmatrix}-\sum_j\partial_j\begin{pmatrix} 0\\ q^{j,\eps}(\bv)
\end{pmatrix}.
\end{equation}
By standard energy estimates, given any $ T> 0$ and $s \geq 2,$
there exists $\delta_*(T) >0,$ such that, if $\| \bv_0\|_{H^s} \leq
\delta_*(T),$ the perturbation system  has a unique solution $\bv \in
{\mathcal C}^0([0,T], H^s(\R\times[-\pi,\pi],\C^n))$, with $\bv(\cdot,0)=\bv_0$, and the bound
$$
\|\bv(\cdot, t)\|_{H^s}\le C\|\bv_0\|_{H^s},
$$
holds for any $t \in [0,T]$ and some $C >0$ that depends on $T$, but
neither on $\e$ nor on $t.$ Likewise, we have a formally quadratic
linearized truncation error $|q^{j,\eps}(\bv)|, |p^{jk,\eps}(\bv)|=O(|\bv|(|\bv|+|\nabla_x\bv|))$ for $\bv\in H^s(\R\times[-\pi,\pi],\C^n)$ with $\|\bv\|_{H^s}\le C$. Our goal in this
section is to establish a quadratic bound on the {\it
linearization error}: \be\label{quaderror} \|\bv(\cdot, T)-e^{L(\eps)
T}\bv_0\|_{H^s}\le C\|\bv_0\|_{H^s}^2. \ee
Here $\{e^{L(\eps)t}\}_{t\geq 0}$ denotes the
$C_0$-semigroup generated by $L(\eps)$, see, e.g., \cite{Lunardi,Z1,Z2}.
This inequality is far from evident in the absence of parabolic
smoothing as shown in \cite{TZ3}.
The corresponding bound does
not hold for quasilinear hyperbolic equations, nor as discussed in
\cite[Appendix A]{TZ3},
for systems of general hyperbolic--parabolic
type, due to loss of derivatives.
However, it follows easily for
systems satisfying assumptions (A1)--(A2).

Applying the differential operator $D^\alpha$ to the perturbation
system and multiplying the result system by
$A^{0,\eps}:=\begin{pmatrix}A^{0}_{11}(\bar \bu^{\eps}) & 0\\0 &
A^{0}_{22}(\bar \bu^{\eps})
\end{pmatrix}:=\begin{pmatrix}A^{0}_{11}& 0\\0 &
A^{0}_{22}
\end{pmatrix}$, we obtain that
\ba\label{A0-mult-eq1}
A^{0,\eps}D^\alpha\partial_t
\bv=&A^{0,\eps}D^\alpha\sum_{jk}\partial_j\Big(B^{jk}(\bar
\bu^{\varepsilon})\partial_k
\bv\Big)+A^{0,\eps}D^\alpha\sum_{jk}\partial_j\Big(\partial_\bu
B^{jk}(\bar \bu^{\varepsilon})\bv\partial_k\bar
\bu^{\varepsilon}\Big)\\&\quad-A^{0,\eps}D^\alpha\sum_j\partial_j (A^{j}\bv)
-A^{0,\eps}D^\alpha\sum_j\partial_j\begin{pmatrix} 0\\ q^{j,\eps}(\bv)
\end{pmatrix}\\
&\quad+A^{0,\eps}D^\alpha\sum_{j,k}\partial_j \begin{pmatrix} 0\\ p^{jk,\eps}(\bv)
\end{pmatrix}.
\ea
To prove our energy estimate we need the following identities:
\ba\label{A0-mult-eq2}
A^{0,\eps}D^\alpha\partial_j\Big(B^{jk}(\bar
\bu^{\varepsilon})\partial_k \bv\Big)&=
\partial_{j}\Big(A^{0,\eps}_{22}b^{jk}(\bar\bu^\eps)D^{\alpha}\partial_k \bv_2
\Big)-(\partial_j A^{0,\eps}_{22})b^{jk}(\bar
\bu^{\eps})D^{\alpha}\partial_k \bv_2\\
&\quad+A^{0,\eps}_{22}\partial_j\Big(\sum_{\beta\leq\alpha;
|\beta|=1}\begin{pmatrix}\alpha\\\beta
\end{pmatrix}D^{\beta}b^{jk}(\bu^\eps)D^{\alpha-\beta}\partial_k \bv_2\Big)\\
&\quad+A^{0,\eps}_{22}\partial_j\Big(\sum_{\beta\leq\alpha;
|\beta|>1}\begin{pmatrix}\alpha\\
\beta
\end{pmatrix}D^{\beta}b^{jk}(\bu^\eps)D^{\alpha-\beta}\partial_k \bv_2\Big);
\ea
\ba\label{A0-mult-eq3}
A^{0,\eps}D^\alpha\partial_j
(A^{j}\bv)=&A^{0,\eps}A^jD^{\alpha}\partial_j \bv
+A^{0,\eps}D^{\alpha}(\partial_j A^j \bv)\\
&\quad+A^{0,\eps}
\sum_{\beta\leq\alpha; |\beta|\geq1}\begin{pmatrix}\alpha\\ \beta
\end{pmatrix}D^{\beta}A^j D^{\alpha-\beta}\partial_j \bv.
\ea
Next, we briefly mention the weak Moser inequality in a channel, another tool needed in our analysis. If    $\kappa\geq1$, $\alpha_1,\dots,\alpha_m$ are multi-indexes, $s=\sum\limits_{i=1}^m|\alpha_i|$ and $h_1,\dots,h_m\in H^{\max\{s,\kappa\}}(\RR\times[-\pi,\pi],\C^n)$, then
\begin{equation}\label{moser}
\|(\partial^{\alpha_1}h_1)\cdot\dots\cdot
(\partial^{\alpha_m}h_m)\|_{L^2}
\leq \Big(\sum^m_{i=1}\|h_i\|_{H^s}\Big)
\Big(\prod_{j\neq i}\|\hat h_j\|_{L^1}\Big)
\leq C\Big(\sum^m_{i=1}\|h_i\|_{H^s}\Big)
\Big(\prod_{j\neq i}\|h_j\|_{H^\kappa}\Big).
\end{equation}
The proof of \eqref{moser} is based on the Hausdorff-Young inequality  and the strong Sobolev embedding principle, see, e.g., \cite[Lemma 1.5]{Z2} for the whole-space case. As an application of the weak Moser inequality, we can prove the following lemma.
\begin{lemma}\label{moserbd}
Assume Hypotheses (A1)--(A2) and (H0). Then,
for any multi-index $\alpha\in\mathbb{N}^2$ with $|\alpha|\geq 2$ and any $\bv\in H^{|\alpha|}(\RR\times[-\pi,\pi],\C^n)$ with $\bv_2\in H^{|\alpha|+1}(\RR\times[-\pi,\pi],\C^r)$, the functions $D^{\alpha}q^{j,\eps}(\bv)$ and $D^{\alpha}p^{jk,\eps}(\bv)$ belong to $L^2(\R\times[-\pi,\pi],\C^n)$ and the following estimate of their norms holds:
$$\max\{\|D^{\alpha}q^{j,\eps}(\bv)\|_{L^2},\|D^{\alpha}p^{jk,\eps}(\bv)\|_{L^2}\}
\lesssim\big(\|\bv\|_{H^{|\alpha|}} +
\|\bv_2\|_{H^{|\alpha|+1}}\big) \big( \|\bv\|_{H^{|\alpha|}}+
\|\bv\|_{H^{|\alpha|}}^{|\alpha|}\big). $$
\end{lemma}
The proof of the lemma follows directly from the week Moser inequality, \eqref{moser}, and the properties of the functions $F^j$ and $b^{jk}$, $j,k\in\{1,2\}$, stated in Section~\ref{assumptions}, see, e.g., \cite{Ta,TZ3, Z2}. The main result of this section reads as follows:
\begin{proposition}\label{energyest}
Assume Hypotheses (A1)--(A2) and (H0). Then for any $2 \le s \le \nu-1$, $T_0>0$,
there exists
some $C=C(T_0)>0$ such that for any $\bv$ satisfying  \eqref{inter1-pert} with
initial data $\bv(\cdot, 0)=\bv_0$ sufficiently small in $H^s(\RR\times[-\pi,\pi],\C^n)$, the following inequalities hold true:
\ba\label{en} \|\bv(\cdot, T)\|_{H^s}^2+\int_0^T \|\bv_2(\cdot,
t)\|_{H^{s+1}}^2\, dt
&\le C\|\bv_0\|_{H^s}^2\quad\mbox{for all}\quad T\in[0,T_0],
\ea
\ba\label{varen} \|\bv(\cdot, T)-e^{L(\eps)T}\bv_0\|_{H^s}&\le
C\|\bv_0\|_{H^s}^2\quad\mbox{for all}\quad T\in[0,T_0]. \ea
\end{proposition}
\begin{proof}
Since $A^{0,\eps}$ is symmetric and positive definite, we obtain that the energy functional
\be\label{def-energy}
\mathcal E(\bv):=\frac{1}{2}\sum_{|\alpha|\leq s} \langle D^{\alpha}\bv,
A^{0,\eps}D^{\alpha} \bv\rangle_{L^2}
\ee
defines a norm equivalent to $\|\cdot\|_{H^s}$, i.e., $\mathcal
E(\cdot)^{1/2}\sim \|\cdot\|_{H^s}$. Using \eqref{A0-mult-eq2}, it follows that
\ba\label{diff-energy}
\partial_t \mathcal{E}(\bv)=\sum_{|\alpha|\leq s} \langle D^{\alpha}\bv,
A^{0,\eps}D^{\alpha} \partial_t \bv\rangle_{L^2}.
\ea
Next, we estimate the right-hand side of \eqref{diff-energy}, using \eqref{pert-sys}, \eqref{A0-mult-eq2} and \eqref{A0-mult-eq3} . We break this long estimate into three separate parts. Using the hypotheses from Section~\ref{assumptions} we have that for any $\bv\in H^s(\RR\times[-\pi,\pi],\CC^n)$ with $\|v\|_{H^s}\leq \delta_0\leq\delta_*(T_0)$ the following holds:
\ba\label{rhs-1}
\sum_{jk}&\sum_{|\alpha|\leq s} \Big\langle D^{\alpha}\bv,
A^{0,\eps}D^\alpha\partial_j\Big(B^{jk}(\bar
\bu^{\varepsilon})\partial_k \bv\Big)\Big\rangle_{L^2} =
\sum_{jk}\sum_{|\alpha|\leq s}\Big\langle D^{\alpha}\bv_2,
\partial_{j}\Big(A^{0,\eps}_{22}b^{jk}(\bar\bu^\eps)D^{\alpha}\partial_k
\bv_2 \Big)\Big\rangle_{L^2}\\
&\qquad\qquad+\delta_0
\|\bv_2\|^2_{H^{s+1}}+C(\delta_0)\|\bv_2\|^2_{H^{s}}+\mathcal{O}(1)\|\bv\|^2_{H^{s}}\\
&=-\sum_{jk}\sum_{|\alpha|\leq s}\Big\langle D^{\alpha}\partial_j \bv_2,
\Big(A^{0,\eps}_{22}b^{jk}(\bar\bu^\eps)D^{\alpha}\partial_k \bv_2
\Big)\Big\rangle_{L^2}+\delta_0
\|\bv_2\|^2_{H^{s+1}}+C(\delta_0)\|\bv_2\|^2_{H^{s}}+\mathcal{O}(1)\|\bv\|^2_{H^{s}}\\
 &\leq -\theta \|\bv_2\|^2_{H^{s+1}}+\delta_0
\|\bv_2\|^2_{H^{s+1}}
+C(\delta_0)\|\bv_2\|^2_{H^{s}}+\mathcal{O}(1)\|\bv\|^2_{H^{s}}\\
&\leq(-\theta+\delta_0)\|\bv_2\|^2_{H^{s+1}}+(C(\delta_0)+\mathcal{O}(1))\|\bv\|^2_{H^{s}}.
\ea
In addition, one readily checks that
\ba\label{rhs-2}
\sum_{jk}\sum_{|\alpha|\leq s} \Big\langle D^{\alpha}\bv,
A^{0,\eps}D^\alpha\partial_j&\Big(\partial_\bu B^{jk}(\bar
\bu^{\varepsilon})\bv\partial_k\bar \bu^{\varepsilon}\Big)\Big\rangle_{L^2}=
\sum_{jk}\sum_{|\alpha|\leq s} \langle D^{\alpha}\bv_2,
A^{0,\eps}_{22}D^\alpha\partial_j\Big(\partial_\bu b^{jk}(\bar
\bu^{\varepsilon})\bv\partial_k\bar
\bu_2^{\varepsilon}\Big)\Big\rangle_{L^2}\\
&= \sum_{jk}\sum_{|\alpha|\leq s} \Big\langle
\partial_j\Big((A^{0,\eps}_{22})^* D^{\alpha}\bv_2\Big), D^\alpha\Big(\partial_\bu
b^{jk}(\bar \bu^{\varepsilon})\bv\partial_k\bar
\bu_2^{\varepsilon}\Big)\Big\rangle_{L^2}\\
&\leq  \mathcal{O}(1)\|\bv\|^2_{H^{s}}+\delta_0
\|\bv_2\|^2_{H^{s+1}}+C(\delta_0)\|\bv\|^2_{H^{s}}.
\ea
Moreover, we infer that
\ba\label{rhs-3}
\sum_{j}&\sum_{|\alpha|\leq s} \langle D^{\alpha}\bv,
A^{0,\eps}D^\alpha\partial_j (A^{j}\bv)\rangle_{L^2}=
\sum_{j}\sum_{|\alpha|\leq s} \langle D^{\alpha}\bv,
A^{0,\eps}A^jD^{\alpha}\partial_j
\bv\rangle_{L^2}+\mathcal{O}(1)\|\bv\|^2_{H^{|\alpha|}}\\
&=\sum_{j}\sum_{|\alpha|\leq s} \langle D^{\alpha}\bv_1,
A_{11}^{0,\eps}A^j_{11}D^{\alpha}\partial_j \bv_1\rangle_{L^2} +\sum_{j}\sum_{|\alpha|\leq s} \langle D^{\alpha}\bv_1,
A_{11}^{0,\eps}A^j_{12}D^{\alpha}\partial_j \bv_2\rangle_{L^2}+\mathcal{O}(1)\|\bv\|^2_{H^{|\alpha|}}\\
&\quad+\sum_{j}\sum_{|\alpha|\leq s} \langle D^{\alpha}\bv_2,
A_{11}^{0,\eps}A^j_{21}D^{\alpha}\partial_j \bv_1\rangle_{L^2} +\sum_{j}\sum_{|\alpha|\leq s} \langle D^{\alpha}\bv_2,
A_{11}^{0,\eps}A^j_{22}D^{\alpha}\partial_j \bv_2\rangle_{L^2}\\
&=\sum_{j}\sum_{|\alpha|\leq s} -\frac{1}{2}\langle D^{\alpha}\bv_1,
\partial_ j(A_{11}^{0,\eps}A^j_{11})D^{\alpha}\bv_1\rangle _{L^2}+\sum_{j}\sum_{|\alpha|\leq s} \langle D^{\alpha}\bv_1,
A_{11}^{0,\eps}A^j_{12}D^{\alpha}\partial_j \bv_2\rangle_{L^2}+\mathcal{O}(1)\|\bv\|^2_{H^{s}}\\
&\quad +\sum_{j}\sum_{|\alpha|\leq s} \langle
\partial_j\Big((A_{11}^{0,\eps}A^j_{21})^*D^{\alpha}\bv_2\Big),
D^{\alpha} \bv_1\rangle _{L^2}+\sum_{j}\sum_{|\alpha|\leq s} \langle D^{\alpha}\bv_2,
A_{11}^{0,\eps}A^j_{22}D^{\alpha}\partial_j
\bv_2\rangle_{L^2}\\
&\leq \mathcal{O}(1)\|\bv_1\|^2_{H^s}+\delta_0\|\bv_2\|^2_{H^{s+1}}+C(\delta_0)\|\bv_1\|^2_{H^s} +\delta_0\|\bv_2\|^2_{H^{s+1}}+C(\delta_0)\|\bv\|^2_{H^s}\\
&\quad+\delta_0\|\bv_2\|^2_{H^{s+1}}+C(\delta_0)\|\bv_2\|^2_{H^s}+\mathcal{O}(1)\|\bv\|^2_{H^{s}}\leq  3\delta_0\|\bv_2\|^2_{H^{s+1}}+C(\delta_0)\|\bv\|^2_{H^s}.
\ea
Next, we introduce the function $Q^\eps: H^s(\R\times [-\pi,\pi],\C^n)\to L^2(\R\times [-\pi,\pi],\C^n) $ defined by
\be\label{def-Qeps}
Q^{\eps}(\bv):= -A^{0,\eps}D^\alpha\sum_j\partial_j\begin{pmatrix} 0\\
q^{j,\eps}(\bv)
\end{pmatrix}+A^{0,\eps}D^\alpha\sum_{j,k}\partial_j \begin{pmatrix} 0\\ p^{jk,\eps}(\bv)
\end{pmatrix}.
\ee
From Lemma~\ref{moserbd}, we conclude that for any $\bv\in H^s(\R\times [-\pi,\pi],\C^n)$ the following estimate holds:
\ba\label{est-Qeps}
\sum_{|\alpha|\leq s}& \langle D^{\alpha}\bv, Q^{\eps}(\bv)\rangle_{L^2}=
\sum_{|\alpha|\leq s} \langle D^{\alpha}\bv_2, Q^{\eps}(\bv)\rangle_{L^2}\\
&\leq \mathcal{O}(1)\|\bv_2\|_{H^{s+1}} \big(\|\bv\|_{H^s} +
\|\bv_2\|_{H^{s+1}}\big) \big( \|\bv\|_{H^s}+ \|\bv\|_{H^s}^s\big).
\ea
Finally, from \eqref{rhs-1}--\eqref{est-Qeps} we obtain that
\ba\label{part-energy-1}
\partial_t\mathcal{E}(\bv)&\leq\mathcal{O}(1)\|\bv_2\|_{H^{s+1}}
\big(\|\bv\|_{H^s} + \|\bv_2\|_{H^{s+1}}\big) \big( \|\bv\|_{H^s}+
\|\bv\|_{H^s}^s\big)\\
&\quad-\theta\|\bv_2\|^2_{H^{s+1}}+5\delta_0\|\bv_2\|^2_{H^{s+1}}+\mathcal{O}(1)\|\bv\|^2_{H^s}.
\ea
We choose $\delta_0>0$ such that $5\delta_0<\theta/2$. So long as
$\|\bv\|_{H^s}$ remains sufficiently small, we infer that
\ba\label{part-energy-2}
\partial_t \CalE(\bv) &\le -(\theta/2) \|\bv_2\|_{H^{s+1}}^2
+O(\|\bv\|_{H^s}^2)\\
&\le -(\theta/2) \|\bv_2\|_{H^{s+1}}^2 +C\CalE(\bv), \ea from which
(\ref{en}) follows by Gronwall's inequality since
$$
\CalE(\bv(T))+ (\theta/2)\int_0^T \|\bv_2(\cdot,t)\|_{H^{s+1}}^2 dt \le
C_2\CalE(\bv_0).
$$
To prove \eqref{varen}, we first note that the error function $E(x,t):=\bv(x, t)-e^{L^{\eps}t}\bv_0(x)$ satisfies the equation
\be\label{eq--error}
\partial_t E=L(\eps)E
-\sum_j\partial_j\begin{pmatrix} 0\\ q^{j,\eps}(\bv)
\end{pmatrix}+\sum_{j,k}\partial_j \begin{pmatrix} 0\\ p^{jk,\eps}(\bv)
\end{pmatrix},
\ee
with initial condition $E(\cdot, 0)=0$. We note that equation \eqref{eq--error} has a structure similar to that of \eqref{pert-sys}. Using the same argument as in \eqref{A0-mult-eq2} and \eqref{A0-mult-eq3} one can show that 
\ba\label{part-energy-3}
A^{0,\eps}D^{\alpha}\partial_t E&=A^{0,\eps}D^{\alpha}\Bigg[\sum_{jk}\partial_j\Big(B^{jk}(\bar
\bu^{\eps})\partial_k
E\Big)+\sum_{jk}\partial_j\Big(\partial_\bu B^{jk}(\bar
\bu^{\eps})E\partial_k\bar
\bu^{\eps}\Big)-\sum_j\partial_j (A^{j}E)\\
&\quad-\sum_j\partial_j\begin{pmatrix} 0\\ q^{j,\eps}(\bv)\big)
\end{pmatrix}+\sum_{j,k}\partial_j \begin{pmatrix} 0\\ p^{jk,\eps}(\bv)
\end{pmatrix}\Bigg]\\
&=\sum_{jk}A^{0,\eps}D^{\alpha}\partial_j\Big(B^{jk}(\bar
\bu^{\eps})\partial_k
E\Big)+\sum_{jk}A^{0,\eps}D^{\alpha}\partial_j\Big(\partial_\bu B^{jk}(\bar
\bu^{\eps})E\partial_k\bar
\bu^{\eps}\Big)\\
&\quad-\sum_jA^{0,\eps}D^{\alpha}\partial_j (A^{j}E)-\sum_jA^{0,\eps}D^{\alpha}\partial_j\begin{pmatrix} 0\\ q^{j,\eps}(\bv)
\end{pmatrix}+\sum_{j,k}A^{0,\eps}D^{\alpha}\partial_j \begin{pmatrix} 0\\ p^{jk,\eps}(\bv)
\end{pmatrix}\\
&= \sum_{jk}A^{0,\eps}D^{\alpha}\partial_j\Big(B^{jk}(\bar
\bu^{\eps})\partial_k
E\Big)+\sum_{jk}A^{0,\eps}D^{\alpha}\partial_j\Big(\partial_\bu B^{jk}(\bar
\bu^{\eps})E\partial_k\bar
\bu^{\eps}\Big)\\
&\quad -\sum_jA^{0,\eps}D^{\alpha}\partial_j (A^{j}E)-Q^{\eps}(\bv).
\ea
From the definition of the energy functional in \eqref{def-energy} we note that
$
\mathcal{E}(E)\sim \|E\|^2_{H^s}$. In addition, its time evolution $\partial_t \mathcal{E}(E)$ satisfies the identity
\ba\label{part-energy-5}
\partial_t \mathcal{E}(E) &=\sum_{|\alpha|\leq s}\langle D^{\alpha}E, A^{0,\eps}D^{\alpha}\partial_t E \rangle_{L^2}=\sum_{|\alpha|\leq s}\Big\langle D^{\alpha}E,  \sum_{jk}A^{0,\eps}D^{\alpha}\partial_j\Big(B^{jk}(\bar\bu^{\eps})\partial_k
E\Big) \Big\rangle_{L^2}\\
&\quad+\sum_{|\alpha|\leq s}\Big\langle D^{\alpha}E,  \sum_{jk}A^{0,\eps}D^{\alpha}\partial_j\Big(\partial_\bu B^{jk}(\bar
\bu^{\eps})E\partial_k\bar
\bu^{\eps}\Big) \Big\rangle_{L^2} -\sum_{|\alpha|\leq s}\langle D^{\alpha}E, \sum_jA^{0,\eps}D^{\alpha}\partial_j (A^{j}E) \rangle_{L^2}\\
& \quad-\sum_{|\alpha|\leq s}\langle D^{\alpha}E, Q^{\eps}(\bv)\rangle_{L^2}.
\ea
To estimate the first three terms in the identity above we argue in the same way as in
\eqref{rhs-1}--\eqref{rhs-3}. The fourth term can be controlled by
\be\label{part-energy-6}
\sum_{|\alpha|\leq s}\langle D^{\alpha}E, -Q^{\eps}(\bv)\rangle_{L^2}\leq  \mathcal{O}(1)\|E_2\|_{H^{s+1}} \big(\|\bv\|_{H^s} +
\|\bv_2\|_{H^{s+1}}\big) \big( \|\bv\|_{H^s}+ \|\bv\|_{H^s}^s\big).
\ee
Combining all of these estimates together and using the weighted Young's inequality, we have, as long as $\|\bv\|_{H^s}$ remains sufficiently small, that
\be\label{part-energy-7}
\partial_t \mathcal{E}(E)\leq-\frac{\theta}{2}\|E_2\|^2_{H^{s+1}}+\mathcal{O}(1)\|E\|^2_{H^{s}}+\mathcal{O}(1)\|\bv\|^2_{H^s}\Big( \|\bv\|^2_{H^s}+\|\bv_2\|^2_{H^{s+1}} \Big).
\ee
Also, since $\|\bv(\cdot,t)\|_{H^s}\leq C\|\bv_0\|_{H^s}$, we obtain that
\be\label{energy-part-8}
\partial_t \mathcal{E}(E)\leq-\frac{\theta}{2}\|E_2\|^2_{H^{s+1}}+\mathcal{O}(1)\|E\|^2_{H^{s}}+\mathcal{O}(1)\|\bv_0\|^2_{H^s}\Big( \|\bv\|^2_{H^s}+\|\bv_2\|^2_{H^{s+1}} \Big).
\ee
Using \eqref{en}, Gronwall's inequality and the fact that the error function satisfies the initial condition $E(\cdot, 0)=0$, we conclude that
$$
\mathcal{E}(E(T))\leq\mathcal{O}(1)\|\bv_0\|^2_{H^s}\int_{0}^{T}\Big( \|\bv(\cdot,t)\|^2_{H^s}+\|\bv_2\|^2_{H^{s+1}}  \Big)\,dt\leq\mathcal{O}(1)\|\bv_0\|^4_{H^s},
$$
which implies that $\|E(\cdot,T)\|_{H^s}\leq \mathcal{O}(1)\|\bv_0\|^2_{H^s}$, proving the lemma.
\end{proof}

\section{O(2) bifurcation for the general case}\label{O2-bif}

In this
and the following
section we prove our main result of this paper, the existence of $O(2)$-Hopf bifurcation under a spectral criterion (Hypothesis $D_\eps$) described in detail below. More precisely, we are looking to prove the existence of periodic solution of equation \eqref{pert-sys}, of period $T>0$, by solving for $T$ as a function of the initial data $\bv(0)$ and the bifurcation parameter $\eps$ in the fixed point equation associated to the return map of \eqref{pert-sys}. Furthermore,
we reduce this infinite dimensional nonlinear system to a finite dimensional system by using the a special variant of the Lyapunov-Schmidt reduction method, introduced in \cite{TZ2} and refined in \cite{TZ3}.

In what follows we are going to consider the operator $L(\eps)$ defined in \eqref{def-Leps} as a second order differential operator from
$H^2(\R\times [-\pi,\pi],\C^n)$ to $L^2(\R\times [-\pi,\pi],\C^n)$. By taking Fourier Transform in
$x_2\in[-\pi,\pi]$, we can identify $L^2(\R\times [-\pi,\pi],\C^n)$ with $\ell^2(\Z,L^2(\R,\C^n))$ and
$H^m(\R\times [-\pi,\pi],\C^n)$ with $\bigoplus_{k\in\Z}\hat{H}_k^m(\R,\C^n)$, $m=1,2$, where the Hilbert space $\hat{H}_k^m(\R,\C^n)$ is the Sobolev space $H^m(\R,\C^n)$, $m=1,2$, with the scalar products
\be\label{Hk1-product}
\langle f,g\rangle_{\hat{H}_k^1(\R,\C^n)}=(1+k^2)\langle f,g\rangle_{L^2}+\langle \partial_{x_1}f,\partial_{x_1}g\rangle_{L^2};
\ee
\be\label{Hk2-product}
\langle f,g\rangle_{\hat{H}_k^2(\R,\C^n)}=(1+k^2+k^4)\langle f,g\rangle_{L^2}+(1+2k^2)\langle \partial_{x_1}f,\partial_{x_1}g\rangle_{L^2}+\langle \partial_{x_1}^2f,\partial_{x_1}^2g\rangle_{L^2}.
\ee
The operator $L(\eps)$ can be identified with $(L_k(\eps))_{k\in\Z}$, where $L_k(\eps):\hat{H}_k^2(\R,\C^n)\to L^2(\R,\C^n)$ are defined by
$L_k(\eps)=\widehat{L(\eps)}(k)$. A simple computation shows that
\begin{equation}\label{Lk-eps-rep}
L_k(\eps)=L_0(\eps)+\rmi kJ(\eps)-k^2\oB^{22}(\eps,x_1),
\end{equation}
where
\be\label{L0}
L_0(\eps)=\partial_{x_1}\big[\oB^{11}(\eps,x_1)\partial_{x_1}-\oA^1(\eps,x_1)\big]\ee
and $J(\eps)$ is the first order operator defined by
\be\label{J}
J(\eps)=\big[  \oB^{12}(\eps,x_1)+\oB^{21}(\eps,x_1) \big]\partial_{x_1}+\oB(\eps,x_1)-\oA^2(\eps,x_1).
\ee
Here the functions $\oA^j,\oB^{jk}:( -\delta,\delta)\times\R\to\R^{n\times n}$ are defined by composing the functions $A^j$ and $B^{jk}$, respectively, with $(\eps,\overline{\bu}^\eps(x_1))$.
\begin{remark}\label{Evans-D}
Since the operator $L_k(\eps)$, $k\in\ZZ$, are one-dimensional differential operators (\eqref{Lk-eps-rep}, \eqref{L0}), we note that its eigenvalues can be
obtained, away from the essential spectrum,
as zeros of the classical Evans function, denoted $D(\cdot,k,\eps)$, see, e.g
\cite{Z1,Z2}.
At the special eigenvalues $\lambda=0$ of $L_0(\eps)$, which are embedded in essential spectra, the zeros of the Evans function carry additional information
determining asymptotic stability \cite{Z1}.
\end{remark}

Next, we set up the general $O(2)$ bifurcation spectral criterion.
In addition to assumptions (A1)--(A3), (B1) and (H0)--(H4) imposed in the previous sections,  we impose the following crucial assumption.

\vspace{0.3cm}

\noindent{\bf Hypothesis $(D_\eps)$}.
We assume that the family of operators $L(\eps)$, $\eps\in(-\delta,\delta)$, satisfies the following conditions
\begin{enumerate}

\item[(i)] For each $\eps$ there exists an open set $\Xi(\eps)$ such that
\begin{equation}\label{(D-epsilon-i)}
\{\lambda\in\C:\mathrm{Re}\lambda\geq 0\}\setminus\{0\}\subset\Xi(\eps)\subseteq\rho_{\mathrm{ess}}(L(\eps));
\end{equation}

\item[(ii)] $\lambda=0$ is a zero of algebraic multiplicity $\ell$ (defined in
(H4)) of the Evans function $D(\cdot,0,\eps)$;
\item[(iii)] There exists a pair of eigenvalues $\lambda_\pm(\eps)=\gamma(\eps)+\rmi\omega(\eps)$ of $L(\eps)$ of multiplicity $2$, for which $\gamma(0)=0$ and $\gamma'(0)>0$, associated with operators $L_{\pm k_*}(\eps)$,
$k_*\neq 0$;
\item [(iv)] Besides $\lambda_\pm(\eps)$, the operators $L_k(\eps)$, $k\in\ZZ\setminus\{0\}$, have no other eigenvalues.
\end{enumerate}

The next step in constructing the $O(2)$ bifurcation is to construct the
time-$T$ evolution map of the
equation \eqref{pert-sys}. An crucial role in this construction is played by the rotational invariance and by the $x_2\to-x_2$ invariance of this equation. More precisely, we note that there exists a non-degenerate rotation group of linear operators $\{R(\theta)\}_{\theta\in\R}$ on $L^2(\RR\times [-\pi,\pi],\C^n)$, with $R(\theta)^*=R(-\theta)$, for all $\theta\in\R$\footnote{We can naturally extend the operator $S$ to the complexification of its domain such that $S^*=S$}, and a bounded symmetry $S$ on $L^2(\RR\times [-\pi,\pi],\R^n)$ satisfying
\begin{equation}\label{inv-RS}
\CalF(\eps;R(\theta)\bu)=R(\theta)\CalF(\eps;\bu),\quad \CalF(\eps;S\bu)=S\CalF(\eps;\bu),
\end{equation}
for all $\bu\in H^2(\RR\times [-\pi,\pi],\R^n)$, $\theta\in\R$ and $\eps\in(-\delta,\delta)$. The group $\{R(\theta\}_{\theta\in\R}$ and $S$ satisfy the following condition
\begin{equation}\label{RS}
R(\theta)S=SR(-\theta) \quad\mbox{for all}\quad\theta\in\R.
\end{equation}
Since $L(\eps)=\frac{\partial\CalF}{\partial\bu}(\eps;\overline{\bu}^\eps)$, from \eqref{inv-RS} we obtain that $H^2(\R\times[-\pi,\pi],\C^n)=\dom(L(\eps))$ is invariant under $S$ and $R(\theta)$, for all $\theta\in\R$ and
\begin{equation}\label{inv-LRS}
L(\eps)R(\theta)=R(\theta)L(\eps), \quad L(\eps)S=SL(\eps)\quad\mbox{for all}\quad\theta\in\R,\eps\in(-\delta,\delta).
\end{equation}
We introduce $G$ the generator of the rotation group $\{R(\theta)\}_{\theta\in\R}$ and we note that $\dom(L(\eps))\subset\dom(G)$ for all $\eps\in(-\delta,\delta)$. Moreover, from \eqref{RS} and \eqref{inv-LRS} and since $\{R(\theta)\}_{\theta\in\R}$ is a rotation group, we infer that
\begin{equation}\label{inv-GLS}
G^*=-G,\quad GL(\eps)=L(\eps)G,\quad GS=-SG\quad \mbox{for all}\quad\eps\in(-\delta,\delta).
\end{equation}
Very important in our reduction are the eigenspaces $\Sigma_\pm(\eps)$ associated to the eigenvalues $\lambda_\pm(\eps)$ of $L(\eps)$. In the next lemma we summarize a few basic properties of these eigenspaces.
\begin{lemma}\label{Sig-GS} For any $\eps\in(-\delta,\delta)$ the following assertions hold true:
\begin{itemize}
\item[(i)] The subspaces $\Sigma_\pm(\eps)$ are invariant under $G$, $S$ and $R(\theta)$ for any $\theta\in\R$;
\item[(ii)] There exits $\ale\in\R\setminus\{0\}$ such that $\sigma(G_{|\Sigma_+(\eps)})=\{\pm\rmi\ale\}$;
\item[(iii)] Let $\bwe\in\Sigma_+(\eps)$ be the eigenfunction (unique up to a scalar multiple) of $G_{|\Sigma_+(\eps)}$ corresponding to the eigenvalue $\rmi\ale$. Then, the eigenspaces $\Sigma_\pm(\eps)$ can be represented as follows:
\begin{equation}\label{Spmw}
\Sigma_+(\eps)=\mathrm{Sp}\{\bwe,S\bwe\},\quad \Sigma_-(\eps)=\mathrm{Sp}\{\overline{\bwe},S\overline{\bwe}\}.
\end{equation}
\end{itemize}
\end{lemma}
\begin{proof} Assertion (i) follows from the fact that the operator $L(\eps)$ commutes with
$R(\theta)$, $S$ and $G$ by \eqref{inv-LRS} and \eqref{inv-GLS}.

(ii) From \eqref{inv-GLS} we have that $G^*=-G$ and since by (i) $\Sigma_+(\eps)$ is invariant under $G$, we infer that $(G_{|\Sigma_+(\eps)})^*=-G_{|\Sigma_+(\eps)}$. It follows that $\sigma(G_{|\Sigma_+(\eps)})\subset\rmi\R$. Since $\dim(\Sigma_+(\eps))=2$, we conclude that there exists $\ale\in\R\setminus\{0\}$ such that $\sigma(G_{|\Sigma_+(\eps)})=\{\pm\rmi\ale\}$.
Taking into account that the group of rotations $\{R(\theta)\}$ is non-degenerate, we infer that $\ker G=\{0\}$, which implies that $\ale\ne 0$, proving (ii).

(iii) We note that since $L(\eps)\bwe=\lambda_+(\eps)\bwe$, from \eqref{inv-LRS} it follows that
$$L(\eps)S\bwe=SL(\eps)\bwe=\lambda_+(\eps)S\bwe,$$
and thus $S\bwe\in\Sigma_+(\eps)$. Next, we will show that $\bwe$ and $S\bwe$ are linearly independent. From \eqref{inv-GLS} we obtain that $G(S\bwe)=-SG\bwe=-\rmi\ale S\bwe$. Thus, $\bwe$ and $S\bwe$ are eigenfunctions of the same operator corresponding to different eigenvalues, which proves that $\bwe$ and $S\bwe$ are linearly independent. Using again that $\dim(\Sigma_+(\eps))=2$ we have that
$\Sigma_+(\eps)=\mathrm{Sp}\{\bwe,S\bwe\}.$
Since $\overline{L(\eps)\bu}=L(\eps)\overline{\bu}$ for any $\bu\in H^2(\R\times [-\pi,\pi],\CC^n)$ we readily infer that
$\Sigma_-(\eps)=\mathrm{Sp}\{\overline{\bwe},S\overline{\bwe}\}.$
\end{proof}
Since $\lambda_\pm(\eps)$ is an eigenvalue of $L(\eps)$ and $\overline{\lambda_+(\eps)}=\lambda_-(\eps)$ for any $\eps\in(-\delta,\delta)$, we know that $\lambda_\pm(\eps)$ are also eigenvalues of $L(\eps)^*$ for any $\eps\in(-\delta,\delta)$. Moreover, if we denote by $\tS_\pm(\eps)$ the eigenspaces associated to eigenvalues $\lambda_\pm(\eps)$ of $L(\eps)^*$, we have that $\dim(\tS_\pm(\eps))=2$ for any $\eps\in(-\delta,\delta)$. The properties satisfied by the eigenspaces $\tS_\pm(\eps)$ are similar to the ones described in Lemma~\ref{Sig-GS} as shown in the lemma below.
\begin{lemma}\label{TildeSig-GS}
For any $\eps\in(-\delta,\delta)$ the following assertions hold true:
\begin{itemize}
\item[(i)] The subspaces $\tS_\pm(\eps)$ are invariant under $G$, $S$ and $R(\theta)$ for any $\theta\in\R$;
\item[(ii)] $\sigma(G_{|\tS_-(\eps)})=\{\pm\rmi\ale\}$. The function $\ale$ is the one introduced in Lemma~\ref{Sig-GS}(ii);
\item[(iii)] Let $\tbwe\in\Sigma_-(\eps)$ be the eigenfunction (unique up to a scalar multiple) of $G_{|\tS_+(\eps)}$ corresponding to the eigenvalue $\rmi\ale$. Without loss of generality we can choose $\tbwe$ such that $\langle\bwe,\tbwe\rangle_{L^2}=1$. Then, the eigenspaces $\tS_\pm(\eps)$ can be represented as follows:
$$
\tS_-(\eps)=\mathrm{Sp}\{\tbwe,S\tbwe\},\quad \tS_+(\eps)=\mathrm{Sp}\{\overline{\tbwe},S\overline{\tbwe}\}.
$$
\end{itemize}
\end{lemma}
\begin{proof} First, we note that by taking adjoint in \eqref{inv-LRS} and \eqref{inv-GLS} we obtain that the operator $L(\eps)^*$ commutes with $R(\theta)$, $S$ and $G$ for any $\theta\in\R$ and $\eps\in(-\delta,\delta)$. Since, in addition $\dim(\tS_\pm(\eps))=2$ for any $\eps\in(-\delta,\delta)$, we can obtain all properties above by using the same arguments we used in Lemma~\ref{Sig-GS}. The only thing left to prove is that the operators $G_{|\Sigma_+(\eps)}$ and $G_{|\tS_+(\eps)}$ have the same eigenvalues. Since $G^*=-G$ and the eigenspace $\tS_+(\eps)$ is invariant under $G$ we have that $(G_{|\tS_+(\eps)})^*=-G_{|\tS_+(\eps)}$, and thus $\sigma(G_{|\tS_+(\eps)})=\{\pm\rmi\beta(\eps)\}$ for some $\beta(\eps)\in\R$. To finish the proof all we need to do is to show that $|\ale|=|\beta(\eps)|$. Indeed, one can readily check that
\[
\rmi\ale\langle\bwe,\tbwe\rangle_{L^2}=\langle G\bwe,\tbwe\rangle_{L^2}
=\langle \bwe,G^*\tbwe\rangle_{L^2}=\langle \bwe,-G\tbwe\rangle_{L^2}=-\rmi\beta(\eps)\langle\bwe,\tbwe\rangle_{L^2},
\]
finishing the proof.
\end{proof}
\begin{remark}\label{new-3.5}
For any $\eps\in(-\delta,\delta)$ and $\theta\in\R$ the following assertions hold true:
\begin{equation}\label{SW}
\langle S\bwe,\tbwe\rangle_{L^2}=0.
\end{equation}
\begin{equation}\label{R-theta-w}
R(\theta)\bwe=e^{\rmi\ale\theta}\bwe;R(\theta)S\bwe=e^{-\rmi\ale\theta}S\bwe;R(\theta)\tbwe=e^{\rmi\ale\theta}\tbwe;
R(\theta)S\tbwe=e^{-\rmi\ale\theta}S\tbwe.
\end{equation}
\end{remark}
\begin{proof}
To prove \eqref{SW} we use \eqref{inv-GLS}, Lemma~\ref{Sig-GS} and Lemma~\ref{TildeSig-GS} as follows: first  we compute
\[
\langle SG\bwe,\tbwe\rangle_{L^2}=\langle S(\rmi\ale\bwe),\tbwe\rangle_{L^2}=\rmi\ale\langle S\bwe,\tbwe\rangle_{L^2}.
\]
In addition,
\begin{align*}
\langle SG\bwe,\tbwe\rangle_{L^2}&=-\langle GS\bwe,\tbwe\rangle_{L^2}=-\langle S\bwe,G^*\tbwe\rangle_{L^2}=-\langle S\bwe,-G\tbwe\rangle_{L^2}\\
&=\langle S\bwe,G\tbwe\rangle_{L^2}=\langle S\bwe,\rmi\ale\tbwe\rangle_{L^2}=-\rmi\ale\langle S\bwe,\tbwe\rangle_{L^2}.
\end{align*}
Using the group property of $\{R(\theta)\}_{\theta\in\R}$, one readily infers \eqref{R-theta-w} from Lemma~\ref{Sig-GS} and Lemma~\ref{TildeSig-GS}.
\end{proof}
Throughout this paper we denote by $\Pi_\pm(\eps)$ the orthogonal projection onto $\Sigma_\pm(\eps)$ parallel to $\tS_\pm(\eps)^\perp$ and by $\Pi(\eps)=I-\Pi_+(\eps)-\Pi_-(\eps)$. Also,
we introduce $\Sigma(\eps):=\Range(\Pi(\eps))$ the complement of $\Sigma_+(\eps)\oplus\Sigma_-(\eps)$. From \eqref{Spmw} and \eqref{SW} we know that the projectors $\Pi_\pm(\eps)$
have the following representation:
\begin{equation}\label{Pipm-rep}
\Pi_+(\eps)\bu=\langle\bu,\tbwe\rangle_{L^2}\bwe+\langle\bu,S\tbwe\rangle_{L^2} S\bwe,\quad
\Pi_-(\eps)\bu=\langle\bu,\overline{\tbwe}\rangle_{L^2}\overline{\bwe}+\langle\bu,S\overline{\tbwe}
\rangle_{L^2} S\overline{\bwe}
\end{equation}
for any $\bu\in L^2(\RR\times [-\pi,\pi],\C^n)$. Using the fact that $\bwe$ and $S\bwe$ are eigenfunctions of $L(\eps)$ associated to the eigenvalue $\lambda_+(\eps)$ and $\tbwe$ and $S\tbwe$ are eigenvalues of $L(\eps)^*$ associated to the eigenvalue $\lambda_-(\eps)=\overline{\lambda_+(\eps)}$ for any $\eps\in(-\delta,\delta)$, one can readily check that
\begin{equation}\label{inv-PiL}
\Pi_\pm(\eps)L(\eps)=L(\eps)\Pi_\pm(\eps),\quad \Pi(\eps)L(\eps)=L(\eps)\Pi(\eps)\quad\mbox{for any}\quad\eps\in(-\delta,\delta).
\end{equation}
Next, we note that equation \eqref{pert-sys} is of the form
\begin{equation}\label{new-pert-sys}
\bv'(t)=L(\eps)\bv(t) +\CalN(\eps;\bv(t)), \quad t\geq 0,
\end{equation}
where the non-linear function $\CalN:(-\delta,\delta)\times H^2(\RR\times [-\pi,\pi],\C^n)\to L^2(\RR\times [-\pi,\pi],\C^n)$ is defined by
\begin{equation}\label{def-nonlinear}
\CalN(\eps;\bu)=\CalF(\eps;\bu)-L(\eps)\bu.
\end{equation}
From \eqref{inv-RS} and \eqref{inv-LRS} we conclude that
\begin{equation}\label{inv-N}
\CalN(\eps;R(\theta)\bu)=R(\theta)\CalN(\eps;\bu),\quad \CalN(\eps;S\bu)=S\CalN(\eps;\bu),
\end{equation}
for all $\bu\in H^2(\RR\times [-\pi,\pi],\R^n)$, $\theta\in\R$ and $\eps\in(-\delta,\delta)$.
To construct the return map of \eqref{pert-sys}, we start by coordinatizing equation \eqref{new-pert-sys} as follows:
\begin{equation}\label{coor-1}
\bv_\pm(t)=\Pi_\pm(\eps)\bv(t), \quad \tbv(t)=\Pi(\eps)\bv(t), \quad t\geq 0,\eps\in(-\delta,\delta).
\end{equation}
We introduce the functions $\phi,\psi:\R_+\to\C$ by $\phi(t)=\langle \bv(t),\tbwe\rangle_{L^2}$ and $\psi(t)=\langle \bv(t),S\tbwe\rangle_{L^2}$. Since we are looking for real-valued solutions of our PDE-system we are interested in finding solution $\bv$ of equation \eqref{new-pert-sys} satisfying $\bv(t)=\overline{\bv(t)}$ for each $t\geq 0$. It follows that
\begin{equation}\label{real-v}
\overline{\phi(t)}=\langle \bv(t),\overline{\tbwe}\rangle_{L^2},\quad \overline{\psi(t)}=\langle \bv(t),S\overline{\tbwe}\rangle_{L^2}\quad\mbox{for all}\quad t\geq 0,
\end{equation}
which implies that
\begin{equation}\label{v-pm}
\bv_-(t)=\overline{\bv_+(t)}\quad\mbox{for all}\quad t\geq 0.
\end{equation}
Next, we rewrite the system \eqref{new-pert-sys} in the new variables $(\phi,\psi,\tbv
)$. From Lemma~\ref{TildeSig-GS}(iii) and Remark~\ref{new-3.5} we conclude that \eqref{new-pert-sys} is equivalent to the system
\begin{equation}\label{int-pert-sys}
\left\{\begin{array}{ll}
\phi'=\lambda_+(\eps)\phi+\langle \CalN(\eps;\phi\bwe+\psi S\bwe+\tbv),\tbwe\rangle_{L^2}\\
\psi'=\lambda_+(\eps)\psi+\langle \CalN(\eps;\phi\bwe+\psi S\bwe+\tbv),S\tbwe\rangle_{L^2}\\
\tbv'=\tL(\eps)\tbv+\Pi(\eps)\CalN(\eps;\phi\bwe+\psi S\bwe+\tbv)\end{array}\right.,
\end{equation}
where $\tL(\eps)=L(\eps)_{|\Range(\Pi(\eps))}$.

Since the linear operator $L(\eps)$ generates a $C_0$-semigroup, (see, e.g., \cite{Lunardi} or \cite{Z2}),  from \eqref{inv-PiL} we infer that $\tL(\eps)$ generates a $C_0$-semigroup. Next, we integrate \eqref{int-pert-sys} in $t\in [0,T]$ using the variation of constants formula to obtain the system
\begin{equation}\label{pert-sys4}
\left\{\begin{array}{ll}
\phi(T)=e^{T\lambda_+(\eps)}\phi(0)+\int_0^Te^{(T-t)\lambda_+(\eps)}\Phi(\phi(s),\psi(s),\tbv(s),\eps)\,ds\\
\psi(T)=e^{T\lambda_+(\eps)}\psi(0)+\int_0^Te^{(T-t)\lambda_+(\eps)}\Psi(\phi(s),\psi(s),\tbv(s),\eps)\,ds\\
\tbv(T)=e^{T\tL(\eps)}\tbv(0)+\int_0^Te^{(T-t)\tL(\eps)}\tilde{\mathcal{V}}(\phi(s),\psi(s),\tbv(s),\eps)\,ds
\end{array}\right..
\end{equation}
Here we denoted by $\{e^{t\tL(\eps)}\}_{t\geq 0}$ the $C_0$-semigroup generated by the operator $\tL(\eps)$. The nonlinearities
$\Phi,\Psi:\C^2\times H^2(\R\times [-\pi,\pi],\C^n)\times (-\delta,\delta)\to\CC$, are defined by
\begin{equation*}
\Phi(z_1,z_2,\tbv,\eps)=\langle \Cal N(\eps;z_1\bwe+z_2 S\bwe+\tbv), \tbwe\rangle_{L^2},\;
\Psi(z_1,z_2,\tbv,\eps)=\langle \Cal N(\eps;z_1\bwe+z_2 S\bwe+\tbv), S\tbwe\rangle_{L^2}.
\end{equation*}
In addition, the nonlinear map $\tilde{\mathcal{V}}:\C^2\times H^2(\R\times [-\pi,\pi],\C^n)\times (-\delta,\delta)\to L^2(\R\times [-\pi,\pi],\C^n)$ is defined by
$$\tilde{\mathcal{V}}(z_1,z_2,\tbv,\eps)=
\Pi(\eps)\Cal N(\eps;z_1\bwe+z_2 S\bwe+\tbv).$$
To prove existence of periodic solutions of period $T$ of \eqref{pert-sys} it is enough to show that we can solve for $T$ in the $T$-return map of \eqref{pert-sys4} in terms of the initial conditions.
This is equivalent with finding fixed points of the period map defined by \eqref{pert-sys4} or with finding zeros of the displacement map $\mathrm{\bf Disp}=(\mathrm{\bf Disp}_1, \mathrm{\bf Disp}_2, \mathrm{\bf Disp}_3):\C^2\times H^2(\R\times [-\pi,\pi],\C^n)\times(-\delta,\delta)\times (0,\infty)\to \C^2\times L^2(\R\times [-\pi,\pi],\C^n)$ defined by
\be\label{def-disp1}
\mathrm{\bf Disp}_1(a_1,a_2,\tbv_0,\eps,T)=(e^{T\lambda_+(\eps)}-1)a_1+
\int_0^Te^{(T-t)\lambda_+(\eps)}\Phi(\phi(s),\psi(s),\tbv(s),\eps)\,ds;
\ee
\be\label{def-disp2}
\mathrm{\bf Disp}_2(a_1,a_2,\tbv_0,\eps,T)=(e^{T\lambda_+(\eps)}-1)a_2+
\int_0^Te^{(T-t)\lambda_+(\eps)}\Psi(\phi(s),\psi(s),\tbv(s),\eps)\,ds;
\ee
\be\label{def-disp4}
\mathrm{\bf Disp}_3(a_1,a_2,\tbv_0,\eps,T)=(e^{T\tL(\eps)}-I)\tbv_0+
\int_0^Te^{(T-t)\tL(\eps)}\tilde{\mathcal{V}}(\phi(s),\psi(s),\tbv(s),\eps)\,ds
\ee
where $(\phi,\psi,\tbv)$ is a solution of \eqref{int-pert-sys} with initial condition $(\phi,\psi,\tbv)(0)=(a_1,a_2, \tbv_0)$.
At this moment it is crucial to eliminate $\tbv_0$ from the system
\begin{equation}\label{pert-sys5}
\mathrm{\bf Disp}(a_1,a_2,\tbv_0,\eps,T)=0
\end{equation}
using a special form of the Lyapunov--Schmidt reduction in order to obtain a finite dimensional system. To apply the Lyapunov-Schmidt reduction method, following \cite{TZ1}--\cite{TZ4}, we need to investigate some of the properties of $\tL(\eps)=L(\eps)_{|\Range(\Pi(\eps))}$.
More precisely, we need to investigate the (right) invertibility of $e^{T\tL(\eps)}-I$ for $T>0$ to be chosen later.  From Hypothesis $(D_\eps)$ we infer that $\sigma(\tL(\eps))\cap\rmi\R=\emptyset$ and  $\tL(\eps)$ has only finitely many eigenvalues with positive real part of finite multiplicity. We introduce the spectral projectors
\be\label{proj-pos-neg}
\Pi^{\mathrm{pos}}(\eps)=\;\mbox{spectral projection of}\;\sigma(\tL(\eps))\cap\{\lambda\in\C:\mathrm{Re}\lambda>0\},\;\;\Pi^{\mathrm{neg}}(\eps)=\Pi(\eps)-\Pi^{\mathrm{pos}}(\eps).
\ee
Moreover, if we define $L^{\mathrm{pos}}(\eps)=L(\eps)_{|\Range(\Pi^{\mathrm{pos}}(\eps))}$ and
$L^{\mathrm{neg}}(\eps)=L(\eps)_{|\Range(\Pi^{\mathrm{neg}}(\eps))}$, from the invariance of the spectral projectors we obtain the following diagonal decomposition on $\Range(\Pi(\eps))=\Range(\Pi^{\mathrm{pos}}(\eps))\oplus\Range(\Pi^{\mathrm{neg}}(\eps))$:
\be\label{L-Pi-pos-neg}
L^\Pi(\eps)=\left(\begin{array}{cc} \pL(\eps) & 0 \\ 0 & \nL(\eps)
\end{array}\right).
\ee
Since $\Range(\Pi^{\mathrm{pos}}(\eps))$ is finite dimensional and $\sigma(\pL(\eps))\subset\{\lambda\in\C:\mathrm{Re}\lambda>0\}$, we infer
\be\label{inver-pos}
e^{T\pL(\eps)}-I\;\mbox{is invertible on}\;  \Range(\Pi^{\mathrm{pos}}(\eps))  \;\mbox{for all}\; T>0, \eps\in (-\delta,\delta).
\ee
Taking again Fourier Transform in $x_2\in [-\pi,\pi]$, we can identify $\nL(\eps)$ with $(\nLk(\eps))_{k\in\Z}$ on $\bigoplus_{k\in\Z}\hat{H}_k^2(\R,\C^n)$, where $\nLk(\eps)=\widehat{\nL(\eps)}(k)$.
In addition, using  that $\Range(\Pi^{\mathrm{pos}}(\eps))$ and $\Sigma_\pm(\eps)$ are finite dimensional spaces we have that there exists $\nZ(\eps)$ a finite subset of $\Z$ such that for any $\eps\in(-\delta,\delta)$, the following assertions hold true:
\be\label{dom-nL1}
\dom(\nLk(\eps))=\hat{H}_k^2(\R,\C^n)\;\mbox{for all}\; k\in\Z\setminus\nZ(\eps);
\ee
\be\label{dom-nL2}
\dom(\nLk(\eps))\;\mbox{is a finite codimension subspace}\; \hat{H}_k^2(\R,\C^n)\;\mbox{for all}\; k\in\nZ(\eps).
\ee
From the definition of spectral projections and Hypothesis $(D_\eps)$ we have that $\sigma(\nL(\eps))\subset\{\lambda\in\C:\mathrm{Re}\lambda<0\}\cup\{0\}$, which implies that
\be\label{Gerh-Pr-1}
\{\lambda\in\C:\mathrm{Re}\lambda>0\}\subset\rho(\nLk(\eps))\;\mbox{for all}\; k\in\Z,\eps\in(-\delta,\delta).
\ee
In the next step we are going to prove that the semigroups generated by $\nLk(\eps)$, $k\in\Z\setminus\{0\}$, are uniformly exponentially stable.
\begin{lemma}\label{Lk-est}
Assume Hypotheses (A1)--(A3), (B1), (H0)--(H4) and $(D_\eps)$. Then, the following assertions hold true:
\begin{enumerate}
\item[(i)] There exists $C>0$ sufficiently large such that for any $k\in\Z\setminus\{0\}, \eps\in(-\delta,\delta)$ the following estimate holds
\be\label{res-k-est}
\|\big(\lambda-\nLk(\eps)\big)^{-1}\|_{\hat{H}_k^1(\R,\C^n)\to\hat{H}_k^1(\R,\C^n)}\leq C\;\mbox{for any}\; \lambda\in\C\;\mbox{with}\;\mathrm{Re}\lambda>0;
\ee
\item[(ii)] There exists a constant $\nu>0$, small enough such that for any $k\in\Z\setminus\{0\}, \eps\in(-\delta,\delta)$ the following estimate holds
\be\label{sem-nLk-est}
\|e^{t\nLk(\eps)}\|_{\hat{H}_k^1(\R,\C^n)\to\hat{H}_k^1(\R,\C^n)}\leq Ce^{-\zeta t}\;\mbox{for any}\; t\geq 0.
\ee
\end{enumerate}
\end{lemma}
\begin{proof} First, we note that we can apply the results from
\cite[Prop. 4.7]{Z2}
to conclude that
there are two positive constants $R$ and $C$ sufficiently large and $\theta_{00}>0$ sufficiently small such that for any $\eps\in (-\delta,\delta)$
\be\label{part-est1}
\|\big(\lambda-L_k(\eps)\big)^{-1}\|_{\hat{H}_k^1(\R,\C^n)\to\hat{H}_k^1(\R,\C^n)}\leq C
\;\mbox{whenever}\;|(k,\lambda)|\geq R,\; \mathrm{Re}\lambda>-\theta_{00}.\ee
It follows that
\be\label{part-est2}
\|\big(\lambda-\nLk(\eps)\big)^{-1}\|_{\hat{H}_k^1(\R,\C^n)\to\hat{H}_k^1(\R,\C^n)}\leq C
\;\mbox{whenever}\;|\lambda|\geq R,\; \mathrm{Re}\lambda\geq 0
\ee
for any $\eps\in (-\delta,\delta)$ and any $k\in\Z\setminus\{0\}$.
From Hypothesis $(D_\eps)$ we conclude that $\sigma(\nL(\eps))=\sigma_{\mathrm{ess}}(L(\eps))$, which implies that the operator $\nL(\eps)$ has only essential spectrum, that is $\sigma(\nLk(\eps))=\sigma_{\mathrm{ess}}(\nLk(\eps))$ for all $k\in\Z\setminus\{0\}$. Moreover, from (A1)--(A3) and (H3) we infer that
\[
\sup\mathrm{Re}\sigma_{\mathrm{ess}}(\nLk(\eps))\leq\sup_{\xi\in\R}\Big[ -\frac{\theta_0(\xi^2+k^2)}{1+\xi^2+k^2}\Big]\leq -\frac{\theta_0}{2}<0
\]
for any $\eps\in(-\delta,\delta)$ and any $k\in\Z\setminus\{0\}$. We conclude that
$\{\lambda\in\C:\mathrm{Re}\lambda\geq 0,\;|\lambda|\leq R\}$  is contained in $\rho(\nLk(\eps))$
for any $\eps\in(-\delta,\delta)$ and any $k\in\Z\setminus\{0\}$, which implies that
\be\label{part-est3}
\|\big(\lambda-\nLk(\eps)\big)^{-1}\|_{\hat{H}_k^1(\R,\C^n)\to\hat{H}_k^1(\R,\C^n)}\leq C
\;\mbox{whenever}\;|\lambda|\leq R,\; \mathrm{Re}\lambda\geq 0
\ee
for any $\eps\in (-\delta,\delta)$ and any $k\in\Z\setminus\{0\}$. Assertion (i) follows shortly from
\eqref{part-est2} and \eqref{part-est3}. Assertion (ii) follows from the Gearhart-$\mathrm{Pr\ddot{u}ss}$ Spectral Mapping theorem for $C_0$-semigroups on Hilbert spaces and the estimate \eqref{res-k-est}.
\end{proof}
The (right) invertibility problem for $e^{T\nLze(\eps)}-I$ was settled in
\cite[Prop.4]{TZ3} (see also \cite[Lemma 5.10]{TZ2}).  To formulate this result we need to introduce the function spaces $X_1$, $B_1$, $X_2$ and $B_2$ as follows:
\be\label{def-X1B1}
X_1=H_\eta^2(\R,\C^n)=H^2(\R,\C^n;e^{\eta(1+|x_1|^2)^{1/2}}dx_1), \quad B_1=H^1(\R,\C^n),
\ee
with their natural Hilbert space scalar product. Furthermore, we define
\be\label{def-X2B2}
X_2=\partial_{x_1}H_{2\eta}^1(\R,\C^n)\cap X_1,\quad B_2=\partial_{x_1}L^1(\R,\C^n)\cap B_1.
\ee
and note that $X_2$ is a Hilbert space while $B_2$ is a Banach space. The scalar product on $X_2$ and
the norm on $B_2$ are defined by
\be\label{def-normsX2B2}
\langle \partial_{x_1}f,\partial_{x_1}g\rangle_{X_2}=\langle \partial_{x_1}f,\partial_{x_1}g\rangle_{X_1}+
\langle f,g\rangle_{H_{2\eta}^1},\quad \|\partial_{x_1}f\|_{B_2}=\|f\|_{L^1}+\|\partial_{x_1}f\|_{B_1}.
\ee
\begin{remark}\label{inv-L0}
Under Hypotheses (A1)--(A3), (B1), (H0)--(H4) and $(D_\eps)$, we can choose a $\delta>0$ small enough such that the operator $e^{T\nLze(\eps)}-I$ has a right inverse, bounded from $X_2$ to $X_1$ and from $B_2$ to $B_1$, uniformly in $(\eps,T)$ for $T\in [T_0,T_1]$, $0<T_0<T_1<\infty$, and $\eps\in(-\delta,\delta)$.  Moreover, the function
\be\label{right-inv-C1}
(\eps,T)\to(e^{T\nLze(\eps)}-I)_{|X_2}^\dagger:(-\delta,\delta)\times [T_0,T_1]\to L(X_2,X_1)\cap L(B_2,B_1)
\ee
is $\mathcal{C}^1$ in the $L(B_2,B_1)$ norm\footnote{Throughout this paper we use $A^\dagger$ to denote the right inverse of linear operator $A$}.
\end{remark}
In the following lemma we collect the results from the previous lemmas on the invertibility of
$e^{T\nL(\eps)}-I$. To formulate the result we define the spaces
\be\label{def-cal-X1B1}
\CalX_1=\bigoplus_{k\in\Z\setminus\{0\}}\hat{H}_k^1(\R,\C^n)\oplus X_1,\quad
\CalB_1=\bigoplus_{k\in\Z\setminus\{0\}}\hat{H}_k^1(\R,\C^n)\oplus B_1,
\ee
\be\label{def-cal-X2B2}
\CalX_2=\bigoplus_{k\in\Z\setminus\{0\}}\hat{H}_k^1(\R,\C^n)\oplus X_2,\quad
\CalB_2=\bigoplus_{k\in\Z\setminus\{0\}}\hat{H}_k^2(\R,\C^n)\oplus B_2.
\ee
We recall the definition of $\hat{H}_k^m(\R,\C^n)$, $k\in\Z$, $m=1,2$, given in \eqref{Hk1-product} and \eqref{Hk2-product}
and the definition of $X_1$, $B_1$, $X_2$ and $B_2$ in \eqref{def-X1B1} and\eqref{def-X2B2}.
\begin{lemma}\label{inv-lem}
Under Hypotheses (A1)--(A3), (B1), (H0)--(H4) and $(D_\eps)$, we can choose $\delta>0$ small enough and $T_0>0$ large enough, such that the operator $e^{T\nL(\eps)}-I$ has a right inverse, bounded  from $\CalX_2$ to $\CalX_1$ and from $\CalB_2$ to $\CalB_1$, uniformly in $(\eps,T)$ for $T\in [T_0,T_1]$, $T_1<\infty$, and $\eps\in(-\delta,\delta)$. Moreover, the function
\be\label{right-inv-cal-C1}
(\eps,T)\to(e^{T\tL(\eps)}-I)_{|\CalX_2}^\dagger:(-\delta,\delta)\times [T_0,T_1]\to L(\CalX_2,\CalX_1)\cap L(\CalB_2,\CalB_1)
\ee
is $\mathcal{C}^1$ in the $L(\CalB_2,\CalB_1)$ norm.
\end{lemma}
\begin{proof} Without loss of generality we can choose $T_0>0$ large enough so that $Ce^{-\zeta T_0}\leq\frac{1}{2}$, where $C$ and $\nu$
are defined in \eqref{res-k-est} and \eqref{sem-nLk-est}.
The invertibility result follows from Lemma~\ref{Lk-est}, Remark~\ref{inv-L0}, \eqref{L-Pi-pos-neg} and \eqref{inver-pos}.
Next, we note that using the regularity properties of $L(\eps)$ we have that
\be\label{C1-good}
(\eps,T)\to(e^{T\nLk(\eps)})_{k\in\Zzer}:(-\delta,\delta)\times [T_0,T_1]\to L\Big(\bigoplus_{k\in\Z\setminus\{0\}}\hat{H}_k^2(\R,\C^n),\bigoplus_{k\in\Z\setminus\{0\}}\hat{H}_k^1(\R,\C^n)\Big).
\ee
is $C^1$. The lemma follows shortly from Remark~\ref{inv-L0}, \eqref{L-Pi-pos-neg}, \eqref{inver-pos} and \eqref{C1-good}.
\end{proof}
Now we are ready to apply the Lyapunov-Schmidt reduction on \eqref{pert-sys5}. We follow the procedure described in detail in \cite{TZ2} and further developed in \cite{TZ3}. Using the result from Lemma~\ref{inv-lem}, we note that the infinite-dimensional part of equation \eqref{pert-sys5}, $\mathrm{\bf Disp}_3(a_1,a_2,\tbv_0,\eps,T)=0$, is equivalent to
\begin{equation}\label{D4-eq}
\tbv_0=(I-e^{T\tL(\eps)})_{|\CalX_2}^\dagger{\bf N}_3(a_1,a_2,\tbv_0,\eps,T)+h,
\end{equation}
for $h\in\ker(I-e^{T\tL(\eps)})\cap\CalX_1$, where ${\bf N}_j$, $j=1,2,3$, denote the nonlinear integral terms from the definition of
$\mathrm{\bf Disp}$ in \eqref{def-disp1}-\eqref{def-disp4}. Another key element of the analysis in \cite{TZ2} is to use the (right) invertibility result from Lemma~\ref{inv-lem} to show that $(I-e^{T\tL(\eps)})_{|\CalX_2}^\dagger$ is bounded on $\Range{\bf N}_3$ and then prove contractivity by
the quadratic bounds of ${\bf N}_3$. In the next lemma we collect some estimates satisfied by the nonlinearities ${\bf N}_j$, $j=1,2,3$.
\begin{lemma}\label{N-estm}
Under Hypotheses (A1)--(A3), (B1), (H0)--(H4) and $(D_\eps)$,  the function ${\bf N}_j:\C^2\times H^2(\R\times [-\pi,\pi],\C^n)\times(-\delta,\delta)\times (0,\infty)\to\CC$, $j=1,2$, is quadratic order and $C^1$. Moreover,
for any $M>0$  the map ${\bf N}_3:\C^2\times B_{\CalX_1}(0,M)\times(-\delta,\delta)\times (0,\infty)\to\CalX_2$ is quadratic order and $C^1$ from
$\C^2\times\CalB_1\times(-\delta,\delta)\times (0,\infty)$ to $\CalB_2$. More precisely the following estimates hold true:
\begin{align*}
&|{\bf N}_j(a_1,a_2,\tbv_0,\eps,T)|+|\partial_{\eps,T}{\bf N}_j(a_1,a_2,\tbv_0,\eps,T)|\leq c(|a_1|+|a_2|+\|\tbv_0\|_{\CalB_1})^2\\
&|\partial_{a_k}{\bf N}_j(a_1,a_2,\tbv_0,\eps,T)|+\|\partial_{\tbv_0}{\bf N}_j(a_1,a_2,\tbv_0,\eps,T)\|_{\CalB_1}\leq c(|a_1|+|a_2|+\|\tbv_0\|_{\CalB_1})\\
&|{\bf N}_3(a_1,a_2,\tbv_0,\eps,T)|+\|\partial_{\eps,T}{\bf N}_3(a_1,a_2,\tbv_0,\eps,T)\|_{\CalB_2}\leq c(|a_1|+|a_2|+\|\tbv_0\|_{\CalX_1})^2\\
&\|\partial_{a_k}{\bf N}_3(a_1,a_2,\tbv_0,\eps,T)\|_{\CalB_2}+\|\partial_{\tbv_0}{\bf N}_3(a_1,a_2,\tbv_0,\eps,T)\|_{L(\CalB_1,\CalB_2)}\leq c(|a_1|+|a_2|+\|\tbv_0\|_{\CalX_1}),
\end{align*}
for any $j=1,2$ and $k=1,2$, whenever $\|\tbv_0\|_{\CalX_1}\leq M$.
\end{lemma}
\begin{proof}
These estimates follow shortly from the estimates from Proposition~\ref{energyest} and the variation of constants formulas of ${\bf N}_j$, given in \eqref{def-disp1}-\eqref{def-disp4}.
\end{proof}
We conclude this section with this a result describing $\ker(I-e^{T\tL(\eps)})$.
\begin{lemma}\label{sem-ker}
Under Hypotheses (A1)--(A3), (B1), (H0)--(H4) and $(D_\eps)$, there exist smooth functions $h_1,\dots,h_\ell:(-\delta,\delta)\to H^2(\R\times [-\pi,\pi],\C^n)$ such that $\{h_j(\eps):j=1,\dots,\ell\}$
is a basis of $\ker(I-e^{T\tL(\eps)})$ of dimension $\ell$ (defined in (H4)) for all $T>0$.
\end{lemma}
\begin{proof}
From hypothesis $(D_\eps)$ (ii) one can readily infer the existence of an $\ell$-dimensional, $\eps$-smooth basis of $\ker\tL(\eps)$. The lemma follows shortly by using the intricate connections between
$\ker\tL(\eps)$ and $\ker(I-e^{T\tL(\eps)})$ that one can readily check by using elementary semigroup theory.
\end{proof}

\section{Proof of the main result}\label{Proof-MT}
In this section we collect the result from the previous sections to prove the existence of an $O(2)$-Hopf bifurcation from our  one-parameter family of standing viscous planar shocks. First, we solve for $\tbv_0$ in the infinite-dimensional equation $\mathrm{\bf Disp}_3=0$.
\begin{lemma}\label{l4.1}
Under Hypotheses (A1)--(A3), (B1), (H0)--(H4) and $(D_\eps)$, there exists $\delta>0$ small enough and a map $\bz:\C^2\times (-\delta,\delta)\times [T_0,T_1]\times\RR^\ell\to\CalX_1$ that is $C^1$ from $\C^2\times (-\delta,\delta)\times [T_0,T_1]\times\RR^\ell$ to $\CalB_1$ such that for any $a_1,a_2\in B_\CC(0,\delta)$, $\eps\in(-\delta,\delta)$ and $T\in [T_0,T_1]$, the local solutions of equation $\mathrm{\bf Disp}_3(a_1,a_2,\tbv_0,\eps, T)=0$ are given by $\tbv_0=\bz(a_1,a_2,\eps,T,b_1,\dots,b_\ell)$ for some $(b_1,\dots,b_\ell)\in\RR^\ell$.
\end{lemma}
\begin{proof} We define the map
$\bF:\C^2\times\CalX_1\times(-\delta,\delta)\times [T_0,T_1]\times\RR^\ell\to\CalX_1$ by
\begin{equation}\label{def-Z}
\bF(a_1,a_2,\tbv_0,\eps,T,b_1,\dots,b_\ell)=(I-e^{T\tL(\eps)})_{|\CalX_2}^\dagger{\bf N}_3(a_1,a_2,\tbv_0,\eps,T)+\sum\limits_{j=1}^{\ell}b_jh_j(\eps).
\end{equation}
From \eqref{D4-eq} we have that the equation $\mathrm{\bf Disp}_3(a_1,a_2,\tbv_0,\eps,T)=0$ is equivalent to the fixed point equation
\begin{equation}\label{D4-Z-eq}
\tbv_0=\bF(a_1,a_2,\tbv_0,\eps,T,b_1,\dots,b_\ell),
\end{equation}
for some $b=(b_1,\dots,b_\ell)\in\RR^\ell$. Using the results from \cite[Section 2]{TZ2} and \cite[Section 4]{TZ3} we can show that the map $\bF$ is bounded from $\C^2\times\CalX_1\times(-\delta,\delta)\times [T_0,T_1]\times\RR^\ell$ to $\CalX_1$ and $C^1$ from $\C^2\times\CalB_1\times(-\delta,\delta)\times [T_0,T_1]\times\RR^\ell$ to $\CalX_1$. Moreover, using the results from Lemma~\ref{inv-lem} and Lemma~\ref{N-estm} we readily obtain appropriate estimates on $\bF$ and its partial derivatives. Using again the results from
\cite[Section 2]{TZ2} we infer that there exists a map $\bz:\C^2\times (-\delta,\delta)\times [T_0,T_1]\times\RR^\ell\to\CalX_1$ that is $C^1$ from $\C^2\times (-\delta,\delta)\times [T_0,T_1]\times\RR^\ell$ to $\CalB_1$ such that $\tbv_0=\bz(a_1,a_2,\eps,T,b_1,\dots,b_\ell)$ solves \eqref{D4-Z-eq} locally, proving the lemma.
\end{proof}
\begin{remark}\label{r4.2}
Since the function $\bz$ is $C^1$ from $\CC^2\times (-\delta,\delta)\times [T_0,T_1]\times\RR^\ell$ to $\CalB_1$, we can use the results from Lemma~\ref{inv-lem} and Lemma~\ref{N-estm} to infer that
\begin{equation}\label{kappa-0}
\bz(0,0,0,T,0,\dots,0)=0.
\end{equation}
Moreover, by differentiating with respect to $T$ in \eqref{D4-Z-eq}, one can easily check that
\be\label{pa-varkappa}
\partial_T\bz(0,0,0,T,0,\dots,0)=0.
\ee
\end{remark}
At this point we note that to solve the equation $\mathrm{\bf Disp}=0$ it is enough to solve a system of two scalar (complex) equations, with variables $a_1,a_2,\eps,T$ and parameters $b_1,\dots,b_\ell$.
Next, we choose $k_*\in\ZZ$, with $|k_*|$ large enough such that $\frac{2 k_*\pi}{\omega(0)}\in (T_0,T_1)$.
Let $T_*:(-\delta,\delta)\to\RR$ be the function defined by $T_*(\eps)=\frac{2k_*\pi}{\omega(\eps)}$.
We plug in
\begin{equation}\label{subst-v,T}
\tbv_0=\bz(a_1,a_2,\eps,T,b_1,\dots,b_\ell)\quad\mbox{and}\quad T=T_*(\eps)(1+\mu),\quad \mu\in(-\delta,\delta)
\end{equation}
in \eqref{def-disp1} and \eqref{def-disp2} to obtain the system
\be\label{tilde-N1,2}
\left\{\begin{array}{ll}
\tilde{\bf N}_1(a_1,a_2,\eps,\mu,b_1,\dots,b_\ell)=0\\
\tilde{\bf N}_2(a_1,a_2,\eps,\mu,b_1,\dots,b_\ell)=0
\end{array}\right..
\ee
To solve the remaining $\CC^2$ system in variables $a_1$, $a_2$, with bifurcation parameters $\eps$ and $\mu$, and involving parameters $(b_1,\dots,b_\ell)\in(-\delta,\delta)^\ell$, we need to identify the symmetries that are satisfied by this system, which are inherited from the original system \eqref{pert-sys} or its reformulation \eqref{int-pert-sys}.
\begin{lemma}\label{l4.3}
Under Hypotheses (A1)--(A3), (B1), (H0)--(H4) and $(D_\eps)$, the finite-dimensional system \eqref{tilde-N1,2} is invariant in the variables $(a_1,a_2)$ under the symmetry $(z_1,z_2)\to(z_2,z_1)$ and the rotation $(z_1,z_2)\to (z_1e^{\rmi\theta},z_2e^{-\rmi\theta})$ for any $\theta\in\RR$.
\end{lemma}
\begin{proof}
We recall that the system \eqref{pert-sys} is invariant under the symmetry $S$ and the non-degenerate rotation group $\{R(\theta)\}_{\theta\in\RR}$. Therefore, the group actions of $S$ and the $C_0$-group $\{R(\theta)\}_{\theta\in\RR}$ are inherited on the eigenspaces $\Sigma_\pm(\eps)$ associated to the crossing eigenvalues $\lambda_\pm(\eps)$. From Lemma~\ref{Sig-GS} and since $S$ is a symmetry, we infer that the group action of $S$ on the eigenspaces $\Sigma_\pm(\eps)$ is isomorphic to the transformation $\tilde{S}:\CC^2\to\CC^2$ defined by $\tilde{S}(z_1,z_2)=(z_2,z_1)$. Using Lemma~\ref{Sig-GS} and \eqref{R-theta-w}, we conclude that for any $\theta\in\RR$ the group action of $R(\theta)$ on the eigenspaces $\Sigma_\pm(\eps)$
 is isomorphic to the transformation $\tilde{R}(\theta):\CC^2\to\CC^2$ defined by $\tilde{R}(\theta)(z_1,z_2)=(z_1e^{\rmi\ale \theta}, z_2e^{-\rmi\ale\theta})$. It follows that the system \eqref{tilde-N1,2} is invariant in the variables $(a_1,a_2)\in\CC^2$ under the transformations $\tilde{S}$ and $\tilde{R}(\theta)$ for all $\theta\in\RR$. Making the change of variables $\theta\to\frac{\theta}{\ale}$, the lemma follows shortly.
\end{proof}
It was shown in \cite[Rmk 13]{TZ3} that one can easily improve the $C^1$-regularity of the map $\bz$ by choosing the spaces $X_j$, j=1,2 (defined in \eqref{def-X1B1} and \eqref{def-X2B2}), appropriately.
\begin{remark}\label{smooth-inv}
By choosing the space $X_1=H_\eta^4(\RR,\CC)=H^4(\R,\C^n;e^{\eta(1+|x_1|^2)^{1/2}}dx_1)$ and
$X_2=\partial_{x_1}H_{2\eta}^1(\R,\C^n)\cap X_1$, one can use the same analysis from \cite{TZ2,TZ3} to prove that $\bz$ is of class $C^2$. More generally,  we can strengthen the result by proving that the map $\bz$ is of class $C^m$ where $\nu=2m+1$ in (H0). Since $\nu\geq 5$, we have that $\bz$ is (at least) of class $C^3$  from $\CC^2\times (-\delta,\delta)\times [T_0,T_1]\times\RR^\ell$ to $\CalB_1$. Therefore, we conclude that the functions $\tilde{\bf N}_j$, $j=1,2$, are of class (at least) $C^3$ from $B_\CC(0,\delta)^2\times(-\delta,\delta)^{2+\ell}$ to $\CC$.
\end{remark}
In the next lemma we exploit the fact that $\tilde{\bf N}_j$, $j=1,2$, are of class $C^3$ by expanding these functions to cubic order.
\begin{lemma}\label{l4.5}
Under Hypotheses (A1)--(A3), (B1), (H0)--(H4) and $(D_\eps)$, there exists $\delta>0$ small enough, two non-zero real constants $\varkappa,\chi\ne 0$ and
smooth functions $\Lambda,\Gamma:(-\delta,\delta)^\ell\to\CC$ such that
\begin{align}\label{structure}
\tilde{\bf N}_1(a_1,a_2,\eps,\mu,b_1,\dots,b_\ell)&=a_1\Bigl(\varkappa\eps+
\rmi\chi\mu
+\Lambda(b_1,\dots,b_\ell)|a_1|^2+\Gamma(b_1,\dots,b_\ell)|a_2|^2\Bigr)+\mathcal{O}(4),\nonumber\\
\tilde{\bf N}_2(a_1,a_2,\eps,\mu,b_1,\dots,b_\ell)&=a_2\Bigl(\varkappa\eps
+
\rmi\chi\mu
+\Gamma(b_1,\dots,b_\ell)|a_1|^2+\Lambda(b_1,\dots,b_\ell)|a_2|^2\Bigr)+\mathcal{O}(4).
\end{align}
\end{lemma}
\begin{proof}
To find an expansion for the functions $\tilde{\bf N}_j$, $j=1,2$, we use the definition of the functions
$\mathrm{\bf Disp}_j$, $j=1,2$ given in \eqref{def-disp1} and \eqref{def-disp2}  and the invariance properties of the system \eqref{tilde-N1,2} proved in Lemma~\ref{l4.3}. From Lemma~\ref{N-estm},
\eqref{kappa-0}, \eqref{pa-varkappa} and the substitution \eqref{subst-v,T} it follows that the leading order terms in $\eps$ and $\mu$ are obtained from the leading order terms of $\mathrm{\bf Disp}_j$, $j=1,2$, namely
\begin{equation}\label{lead-term-Disp}
\Bigl(e^{T_*(\eps)(1+\mu)\lambda_+(\eps)}-1\Bigr)a_j=\Bigl(e^{T_*(\eps)(1+\mu)\gamma(\eps)+2k_*\pi\rmi\mu}-1\Bigr)a_j,\quad j=1,2.
\end{equation}
Taking $\varkappa=\gamma'(0)T_*(0)=\frac{2k_*\pi\gamma'(0)}{\omega(0)}\ne 0$ (by Hypothesis $(D_\eps)$) and $\chi=2k_*\pi\ne 0$, we infer that there exist smooth functions $\Lambda_j, \Gamma_j, \Upsilon_j:(-\delta,\delta)^\ell\to\RR$ such that
\begin{align*}
&\tilde{\bf N}_1(a_1,a_2,\eps,\mu,b_1,\dots,b_\ell)=a_1\Bigl(\varkappa\eps+
\rmi\chi\mu
+\Lambda_1(b_1,\dots,b_\ell)|a_1|^2+\Upsilon_1(b_1,\dots,b_\ell)a_1a_2\nonumber\\
&+\Gamma_1(b_1,\dots,b_\ell)|a_2|^2\Bigr)+\overline{a_2}\Bigl(
+\Lambda_2(b_1,\dots,b_\ell)|a_1|^2+\Upsilon_2(b_1,\dots,b_\ell)\overline{a_1}\overline{a_2}+\Gamma_2(b_1,\dots,b_\ell)|a_2|^2
\Bigr)+\mathcal{O}(4),\nonumber\\
&\tilde{\bf N}_2(a_1,a_2,\eps,\mu,b_1,\dots,b_\ell)=a_2\Bigl(\varkappa\eps
+
\rmi\chi\mu
+\Gamma_1(b_1,\dots,b_\ell)|a_1|^2+\Upsilon_1(b_1,\dots,b_\ell)a_1a_2\nonumber\\
&+\Lambda_1(b_1,\dots,b_\ell)|a_2|^2\Bigr)
+\overline{a_1}\Bigl(
\Gamma_2(b_1,\dots,b_\ell)|a_1|^2+\Upsilon_2(b_1,\dots,b_\ell)\overline{a_1}\overline{a_2}+\Lambda_2(b_1,\dots,b_\ell)|a_2|^2
\Bigr)
+\mathcal{O}(4).
\end{align*}
Noting that there always exist reflectionally symmetric solutions
$a_1\equiv a_2$, for which the situation reduces to that of a standard
Hopf bifurcation, recalling that the system originates from a rotating evolutionary system in which $a_1$, $a_2$ rotate in a common direction with
common (to linear order) speed $\omega(\eps)$, and noting that periodic solutions are preserved under influence of the flow, we infer that $\Upsilon_1=\Lambda_2=\Upsilon_2=\Gamma_2=0$ must hold.
For, otherwise it is easy to see that solutions $a_1\equiv a_2$ are not preserved under an approximate common rotation. Dropping the index of the functions $\Lambda_1$ and $\Gamma_1$, the lemma follows immediately.
\end{proof}
To finish the proof of our main result we introduce the following genericity assumption:
\begin{hypothesis}\label{genericity}
If $\Lambda,\Gamma:(-\delta,\delta)^\ell\to\CC$ are the functions from expansion \eqref{structure}, we assume that there exists $\delta>0$ small enough such that
\begin{equation}\label{geneq}
\Lambda\ne\Gamma,\quad\mathrm{Re}(\Lambda+\Gamma)\ne0,\quad \mathrm{Re}\Lambda\ne 0.
\end{equation}
\end{hypothesis}

\begin{proof}[Proof of Theorem~\ref{main-result}]
From Lemma~\ref{l4.1} it follows that to prove the theorem it is enough to solve the finite-dimensional system \eqref{tilde-N1,2}.
Let $\bb=(b_1,\dots,b_\ell)\in(-\delta,\delta)^\ell$ with $|\bb|\leq C|(a_1,a_2)|$
for some constant $C>0$.
Making the substitution
\begin{equation}\label{rho-substitution}
a_1=a,\quad a_2=\rho a,\quad a\in\CC\quad\mbox{and}\quad \rho\in\CC\cup\{\infty\}.
\end{equation}
in \eqref{tilde-N1,2} and using the results from Lemma~\ref{l4.5}, we obtain the equivalent system:
\begin{equation}\label{new-system-bif}
\left\{\begin{array}{ll}
a\Bigl(\varkappa\eps+\rmi\chi\mu+\Lambda(\bb)|a|^2+\Gamma(\bb)|a|^2|\rho|^2\Bigr)+\mathcal{O}(4)=0\\
a\rho\Bigl(\varkappa\eps+\rmi\chi\mu+\Gamma(\bb)|a|^2+\Lambda(\bb)|a|^2|\rho|^2\Bigr)+\mathcal{O}(4)=0
\end{array}\right..
\end{equation}
There are three cases of interest: when $a\ne 0$ and $\rho$ is bounded and bounded away from $0$ and when $|\rho|<<1$ or $|\rho|>>1$.
We are going to treat all of these cases separately.

\noindent{\bf Case 1.} $a\ne 0$ and there exists $C>0$ such that $\frac{1}{C}\leq |\rho|\leq C$.

\noindent Multiplying the first equation of the system \eqref{new-system-bif} by $\rho$ and subtracting it from the second equation, we obtain the equation:
\begin{equation*}
a\rho\Bigl((\Lambda(\bb)-\Gamma(\bb)\Bigr)|a|^2(1-|\rho|^2)+\mathcal{O}(4)=0.
\end{equation*}
Since in this case $\rho$ is bounded and bounded away from $0$ we can divide this equation by $a\rho|a|^2$ to obtain the equation:
\begin{equation}\label{proof-1}
\bigl((\Lambda(\bb)-\Gamma(\bb)\bigr)(1-|\rho|^2)+\mathcal{O}(a)=0.
\end{equation}
Since by Hypothesis~\ref{genericity} we have that $\Lambda(\bb)\ne\Gamma(\bb)$ from \eqref{proof-1} we infer that
$|\rho|=1+\mathcal{O}(a)$. Substituting back in the first equation of \eqref{new-system-bif} and using that $a\ne 0$ we obtain that
\be\label{proof-2}
\varkappa\eps+\rmi\chi\mu+\bigl(\Lambda(\bb)+\Gamma(\bb)\bigr)|a|^2+\mathcal{O}(a^3)=0.
\ee
Depending on model parameters (specifically, the relative signs of $\varkappa$
and $\Re(\Lambda+\Gamma)$), this will occur for $\eps$ positive (supercritical case) or negative (subcritical case).
Taking the real part in the above equation and since $\varkappa\ne 0$ by Lemma~\ref{l4.5} and $\mathrm{Re}(\Lambda(\bb)+\Gamma(\bb))\ne 0$ by Hypothesis~\ref{genericity}, we infer that $|a|=\mathcal{O}(|\eps|^{1/2})$. From Lemma~\ref{l4.5} we have that $\chi\ne0$, therefore by taking imaginary part in \eqref{proof-2} we conclude that $\mu=\mathcal{O}(\eps)$.

\noindent{\bf Case 2.} $a\ne 0$, $|\rho|<<1$.

\noindent From the first equation of \eqref{new-system-bif}, by using the fact that $\rho$ is small, we obtain the equation
\be\label{proof-3}
\varkappa\eps+\rmi\chi\mu+\bigl(\Lambda(\bb)+|\rho|^2\Gamma(\bb)\bigr)|a|^2+\mathcal{O}(a^3)=0.
\ee
Since $\mathrm{Re}\Lambda(\bb)\ne 0$, by Hypothesis~\ref{genericity}, we have that $\mathrm{Re}\bigl(\Lambda(\bb)+|\rho|^2\Gamma(\bb)\bigr)\ne 0$, by continuity. Taking real and imaginary part in \eqref{proof-3} and using that $\varkappa,\chi\ne0$, we infer that $|a|=\mathcal{O}(|\eps|^{1/2})$ and $\mu=\mathcal{O}(\eps)$.

\noindent{\bf Case 3.} $a\ne 0$, $|\rho|>>1$.

\noindent From the second equation of \eqref{new-system-bif} and since $1/\rho$ is small, we obtain the equation
\be\label{proof-4}
\varkappa\eps+\rmi\chi\mu+\Bigl(\Lambda(\bb)+\frac{1}{|\rho|^2}\Gamma(\bb)\Bigr)|a|^2+\mathcal{O}(a^3)=0.
\ee
Using Hypothesis~\ref{genericity} again, we have that $\mathrm{Re}\Bigl(\Lambda(\bb)+\frac{1}{|\rho|^2}\Gamma(\bb)\Bigr)\ne 0$, by continuity. Arguing just as in the previous cases, by taking real and imaginary part in \eqref{proof-4}, we conclude again that $|a|=\mathcal{O}(|\eps|^{1/2})$ and $\mu=\mathcal{O}(\eps)$.
Summarizing, we conclude that any
solution of \eqref{tilde-N1,2} satisfies one of the conditions
\begin{align}\label{proof-5}
(i)\; |a_1|&=\mathcal{O}(|\eps|^{1/2})\;\mbox{and}\;a_2= e^{\rmi\theta } a_1+\mathcal{O}(a_1^2),\nonumber\\
\;(ii)\;
|a_1|=\mathcal{O}(|\eps|^{1/2})&\;\mbox{and}\;|a_2|<<|a_1|\;\;\;(iii)\; |a_2|=\mathcal{O}(|\eps|^{1/2})\;\mbox{ and}\;|a_1|<<|a_2|.
\end{align}
Moreover, whenever we have a solution of \eqref{tilde-N1,2} we have that $\mu=\mathcal{O}(\eps)$.
Making the change of variables
\be\label{tilda-substitution}
\ta_1=\frac{a_1}{|\eps|^{1/2}},\quad \ta_2=\frac{a_2}{|\eps|^{1/2}},\quad\tmu=\frac{\mu}{\eps}
\ee
in \eqref{tilde-N1,2} we obtain the system:
\begin{equation}\label{new-system-bif2}
\left\{\begin{array}{ll}
\ta_1\Bigl(\varkappa\sgn(\eps)+\rmi\chi\tmu+\Lambda(\bb)|\ta_1|^2+\Gamma(\bb)|\ta_2|^2\Bigr)+\mathcal{O}(|\eps|^{1/2})=0\\
\ta_2\Bigl(\varkappa\sgn(\eps)+\rmi\chi\tmu+\Gamma(\bb)|\ta_1|^2+\Lambda(\bb)|\ta_2|^2\Bigr)+\mathcal{O}(|\eps|^{1/2})=0
\end{array}\right..
\end{equation}
Since the functions $\tilde{\bf N}_j$, $j=1,2$ are of class $C^3$, by Remark~\ref{smooth-inv}, it follows that $\mathcal{O}(|\eps|^{1/2})$ terms in \eqref{new-system-bif2} are of $C^3$ class in $\ta_1,\ta_2$ and $\tmu$.  Since system \eqref{new-system-bif2} is equivalent to the system \eqref{tilde-N1,2} which is rotationally invariant by Lemma~\ref{l4.3},
it inherits the property of rotational invariance, and in addition the property that
$(\tilde a_1, \tilde a_2)=(0,0)$ is always a solution.

Setting $\eps=0$ in \eqref{new-system-bif2},
we reduce to (a rescaled version of)
the truncated cubic system discussed in \S \ref{dtreat},
which has under nondegeneracy assumptions \eqref{geneq}
precisely four families of solutions:
$$
(\tilde a_1,\tilde a_2)= (0,0), \; (a_*,0), \; (0,a_*),
\;(a_{\natural},e^{\rmi\theta }a_{\natural}),
$$
$\theta \in \R$, of which the
last three are nontrivial equilibria
bifurcating from the first, zero equilibrium.
Depending on model parameters, these will occur variously
for $\eps>0$ (supercritical case) or $\eps<0$ (subcritical case),
according to the rule
$\sgn(\eps) = -(\Re \Lambda(\bb) |\tilde a_1|^2 + \Re \Gamma(\bb) |\tilde a_2|^2)/\varkappa$, where
the righthand side is, variously,
$-(\Re \Lambda(\bb) + \Re \Gamma(\bb))a_*^2/\varkappa$,   $-(\Re \Lambda(\bb))/\varkappa$,  or  $-(\Re
\Gamma(\bb)/\varkappa$.
Fixing $\bb=0$, we find by a combination of Implicit Function arguments and
symmetry considerations that each of these families continues uniquely
under small perturbations in $|\eps|^{1/2}$, to give the claimed exact solutions
of the full system obtained by Lyapunov--Schmidt reduction
using the choice of right inverse corresponding to $\bb=0$.
These computations are carried out in Appendix \ref{sec:jacobian}.

It remains to deal with the $\ell$-fold indeterminacy associated with
parameter $\bb$, induced by the $\ell$-fold kernel of the one-dimensional
linearized operator $L_0(\eps)$.
This may be accounted for as follows,
using the method introduced for that purpose in \cite{TZ2}.
First, we make the change of variables
$\tilde{b}_j=\frac{b_j}{|(a_1,a_2)|}$, $j=1,2$, ensuring for fixed
$\tilde \bb$ that $|\bb|\leq C|(a_1,a_2)|$ as required for our arguments above;
by the same arguments, we obtain thereby a unique family of solutions perturbing
from the $\eps=0$ solutions for each choice of $\tilde \bb$, thus obtaining
an $\ell$-dimensional cone of distinct solutions above each solution for
$\bb=0$, which are the unique solutions lying
within the cone $\{(a_1,a_2,\tbv_0):
\|\tbv_0\|_{\CalX_1}\leq C|(a_1,a_2)|\}$, for some $C>0$;
moreover, these cones may be extended to smooth $\ell$-dimensional
manifolds of solutions by perturbing about different choices of background wave in the $\ell$-dimensional manifold of stationary shock solutions with the
same endpoints, thus obtaining a family of nearby problems to which the same arguments uniformly apply.
Finally, applying
\cite[Prop. 2.20]{TZ2} we find that {\it any} periodic solution of \eqref{int-pert-sys} can be shifted by such a change of coordinates so as
to originate in the cone $\{(a_1,a_2,\tbv_0):
\|\tbv_0\|_{\CalX_1}\leq C|(a_1,a_2)|\}$, for some such nearby
problem. Thus, we can infer uniqueness of the full $\ell$-parameter families
of solutions as in
\cite[Cor. 2.21]{TZ2}.
%

{\bf Characterization as traveling waves.}
Finally, we verify in the Lax case, $\ell=1$, that ``traveling-wave'' type
solutions close to $(\tilde a_1,\tilde a_2)= (a_*,0)$
or $(0,a_*)$ are indeed traveling waves with respect to $x_2$.
This may be seen by the computation carried out in Appendix \ref{sec:jacobian}
showing that these types of solutions are unique up to rotational invariance,
i.e., up to translation with respect to $x_2$, for each fixed $\bb$.
On the other hand, time-translates of any such solution give a nearby periodic solution (specifically, nearby in rescaled coordinates), which must therefore be a ``traveling-type'' solution for some nearby choice of $\bb$.
Recalling, for the Lax case $\ell=1$, that change in $\bb$ corresponds
simply to translation in $x_1$, we find therefore that time-translates of
``traveling-type'' solutions correspond to translates in $x=(x_1,x_2)$,
from which it is readily seen (by substitution in the original pde)
that they are traveling waves $h^\eps(x_1-ct,x_2-dt)$ in  $x_1$ and $x_2$.
But, time-periodicity, plus the fact that the background standing shocks
by our assumptions are not periodic in $x_1$, implies that
$c$
must vanish, leaving the conclusion that the
solution must be a traveling wave in $x_2$ alone.
\end{proof}
\begin{remark}\label{r4.6}
The above displacement map argument substitutes in our setting
for the standard approaches used to treat $O(2)$ Hopf bifurcation
in other settings, namely,
the center-manifold/normal forms approach, not available to us because in absence of spectral gap we have no readily available center manifold (possibly even nonexisting as far as we know);
and the spatial dynamics approach,
or Lyapunov-Schmidt reduction of the evolution equations recast
the in the class of time-periodic functions, as described for example
in \cite{GSS}, again not applicable to our framework
based on the time-$T$ map and the ``reverse temporal dynamics'' approach of \cite{TZ3}.
Both of these standard approaches rely on $O(2)\times S^1$ symmetry, with the
additional $S^1$ symmetry imposed, respectively, by normal form reduction
and time-periodicity.
It is interesting that we do not need full $S^1$ symmetry in our argument,
substituting similar but less detailed information coming from the origins of the time-$T$ map as an evolution problem to obtain the key property $\Upsilon_1=\Lambda_2=\Upsilon_2=\Gamma_2=0$.
Though we do not treat it in our analysis, spectral stability information for the bifurcating waves
should in principle be readily available from the reduced
time-$T$ solution map \eqref{tilde-N1,2} with $\eps,\mu$ held fixed.
\end{remark}

\section{Gas dynamics and MHD}\label{s:gas}
Finally, we comment briefly on the cases of gas dynamics or MHD,
to which our analysis does not apply.
In Eulerian coordinates, the 2-D
compressible Navier--Stokes equations are \cite{Ba,Da,Sm}:
\ba\label{eq:ns}
\rho_t+ (\rho u)_x +(\rho v)_y&=0,\\
(\rho u)_t+ (\rho u^2+p)_x + (\rho uv)_y&=
(2\mu +\eta) u_{xx}+ \mu u_{yy} +(\mu+ \eta )v_{xy},\\
(\rho v)_t+ (\rho uv)_x + (\rho v^2+p)_y&=
\mu v_{xx}+ (2\mu +\eta) v_{yy} + (\mu+\eta) u_{yx},\\
(\rho E)_t+ (\rho uE+up)_x + (\rho vE+vp)_y
&=\Big( \kappa T_x + (2\mu+\eta)uu_x + \mu v(v_x+u_y) + \eta uv_y\Big)_x \\
&\quad +\Big( \kappa T_y+ (2\mu+\eta)vv_y + \mu u(v_x+u_y) + \eta vu_x\Big)_y,
\ea
where $\rho$ is density, $u$ and $v$ are the fluid velocities in $x$ and $y$
directions, $p$ is pressure, $T$ is temperature,
$E=e+\frac{u^2}{2} +\frac{v^2}{2}$
is specific energy,
$e$ is specific internal  energy,
$\frac{u^2}{2} +\frac{v^2}{2}$ is kinetic energy,
and the constants $\mu>|\eta|\ge0$ and $\kappa>0$
are coefficients of first (``dynamic'')
and second viscosity and heat conductivity.
The equations are closed by equations of state
\begin{equation}\label{eq:general_eos}
p=p(\rho,T),\quad e=e(\rho,T).
\end{equation}

Here, $x$ and $y$, are coordinates in a rest frame, and $t\ge 0$
is time.
For common fluids and gases in normal conditions,
the polytropic gas laws $p=\Gamma\rho e$, $e=C_v T$
give a good fit to experimental observations,
where $\Gamma>0$ and $C_v>0$ are constants depending on the gas
\cite{Ba}.
The MHD equations feature an additional coupling to a magnetic field vector;
our discussion applies also in that case.

\subsection{Lagrangian formulation}
Equations \eqref{eq:ns} may be converted to Lagrangian coordinates
as follows \cite{A,Da,DM,HT}.
Let $X, Y$ denote a reference configuration, and $T=t$, and
define particle paths $(x,y)(X,Y,t)$ by
$\partial_t (x,y)= (u,v)$,
$(x,y)|_{t=0}=(x_0,y_0)(X,Y)$ for some choice of $(x_0,y_0)(\cdot)$
Denote by $\chi_{jk}$ the entries of
$$
\Omega:=\frac{\partial (x,y)}{\partial (X,Y)}:=\bp \frac{\partial x}{\partial X} & \frac{\partial x}{\partial Y}\\
\frac{\partial y}{\partial X} & \frac{\partial y}{\partial Y} \ep.
$$
A great simplification is obtained if one can choose
$(x_0,y_0)$ so that $\det \Omega|_{t=0}\equiv (1/\rho)|_{t=0}$,
as can always be done for $L^1_{\mathrm{loc}}$ data in 1-D or an
$L^1$ multi-D perturbation thereof- in particular for the perturbed planar shocks considered here.
In this case, the values $(x,y)(X,Y,t)$ may be considered as deformations of an initially uniform-density rest configuration, and equations \eqref{eq:ns}
reduce to a special case of thermoviscoelasticity, with $e$ depending
only on $\tau:=\det\Omega=1/\rho$ and entropy $S$:
\ba\label{eq:tvisc}
(\Omega_{11})_t- u_X&=0,\\
(\Omega_{12})_t- u_Y&=0 ,\\
(\Omega_{21})_t- v_X&=0 ,\\
(\Omega_{22})_t- v_Y&=0,\\
u_t + (\Omega_{22} p)_X-(\Omega_{21} p)_Y &=  \hbox{\rm $2$nd-order derivative terms},\\
v_t -(\Omega_{12} p)_X +(\Omega_{11} p)_Y &= \hbox{\rm $2$nd-order derivative terms},\\
E_t -(\Omega_{12} vp)_X +(\Omega_{11} vp)_Y + (\Omega_{22} up)_X-(\Omega_{21} up)_Y
&= \hbox{\rm $2$nd-order derivative terms},\\
\ea
together with the compatibility conditions (preserved by time evolution):
\ba\label{dfree2}
\partial_Y \Omega_{11}-\partial_X \Omega_{12}&=0
;\quad
\partial_Y \Omega_{21}-\partial_X \Omega_{22}&=0.\\
\ea
Here, we have omitted the description of the
complicated (but divergence-form) righthand sides involving second-order derivative terms involving transport effects, not needed
for our discussion.

\subsection{Coordinate-ambiguity}\label{s:ambiguity}
The augmented system approach followed in this paper
works for the difficult case of thermoviscoelasticity, {\it but it does not work for the apparently simpler case of gas dynamics.}
The reason: for gas dynamics, the ``contingent'' entropy of the enlarged system is only nonstrictly convex and the equations fail to be symmetric, etc.
At an operational level, this is because
there is no penalty on shear strains, giving neutral directions in the
associated entropy function.
In fact, the problem is deeper than that- it reflects the fact that Lagrangian gas dynamic equations have an infinite-dimensional family of invariances consisting of all volume preserving maps of the spatial coordinate.  (Since pressure depends only on
density= reciprocal of the Jacobian of the deformation map, there can be
no change under Jacobian-preserving transformations.)

This massive ambiguity in Lagrangian coordinatization means that there is a corresponding infinite-dimensional family of neutral perturbations in the stress tensor $\Omega$ that do not decay to zero time-asymptotically, but remain constant
without affecting the evolution of other, gas-dynamical, variables, and,
as a consequence, there can be no coordinate system, extended or otherwise,
in which perturbations of $\Omega$ decay.
Thus, the strategy followed here will fail; a side-consequence is that gas-dynamical shocks are {\it never} asymptotically orbitally stable in Lagrangian
coordinates, a fundamental distinction between Eulerian and Lagrangian formulations.

\subsection{Possible remedies}\label{s:remedies}
The treatment of gas dynamics/MHD is an important direction for future
investigation,
perhaps by ``factoring out'' invariances in linearized estimates, then trying to show periodicity modulo these invariant transformations.
Another approach might be to treat the problem instead by spatial dynamics techniques
plus the standard $O(2)$ reduction argument on the space of time-periodic functions.
Here, one must confront similar issues of regularity as faced here,
but without having time-evolutionary stability machinery as a guide.
A further idea
is to carry out a Nash-Moser type iteration in Eulerian coordinates, dealing
with loss of regularity in the quasilinear hyperbolic modes by the iteration scheme instead of re-coordinatization.

Finally, an alternative approach modifying our present method would be
to establish a ``Korn-type'' inequality showing that there exists a
volume-preserving transformation under which $|\Omega|$ is controlled by
$|\det \Omega|$ in $H^s$ norms.
For, we could then carry out our bifurcation analysis
in the usual (Eulerian) gas dynamics variables, including $\tau=\det \Omega$,
choosing an optimal Lagrangian coordinatization in order to control nonlinear
variational estimates on the time-$T$ evolution map.
However, we have been unable to establish such an inequality, and suspect that
one does not hold.
The resolution of this question seems an interesting mathematical problem in its own right.

\appendix
\section{Jacobian computations for the truncated cubic system}\label{sec:jacobian}

Relabeling slightly for notational convenience, we may write the
rescaled system \eqref{new-system-bif2} as
\ba\label{original}\left\{\begin{array}{ll}
(\varkappa\sgn(\eps)  + \rmi\chi\tmu + \Lambda |a|^2+ \Gamma |h|^2)a=|\eps|^{1/2}\Phi_1(\eps, \tmu,a,h),\\
(\varkappa\sgn(\eps) + \rmi\chi\tmu + \Lambda |h|^2+ \Gamma |a|^2)h=|\eps|^{1/2} \Phi_2(\eps, \tmu,a,h),\\
\end{array}\right.\ea
where $\varkappa,\chi$ is constant, $\tmu$ is real, $a$ and $h$ complex, and
$\Phi_j\in \CC$ are $C^1$ functions, taking
without loss of generality (using rotational invariance)
$a=(a_1,0)$, $h=(h_1,h_2)$
and assuming the genericity conditions
\begin{equation}\label{hyp}
\Lambda\ne\Gamma,\quad\mathrm{Re}(\Lambda+\Gamma)\ne0,\quad \mathrm{Re}\Lambda\ne 0.
\end{equation}

Expressed in real coordinates, these are:
\be\label{real}
\left\{\begin{array}{ll}
(\varkappa\sgn(\eps)  + \Re \Lambda |a|^2+ \Re \Gamma |h|^2)a_1=|\eps|^{1/2} \Re \Phi_1(\eps, \tmu,a,h),\\
(\chi\tmu + \Im \Lambda |a|^2+ \Im \Gamma |h|^2)a_1=|\eps|^{1/2}\Im \Phi_1(\eps, \tmu,a,h),\\
(\varkappa\sgn(\eps) +  \Re\Lambda |h|^2+ \Re \Gamma |a|^2)h_1 -
(\chi\tmu + \Im\Lambda |h|^2+ \Im\Gamma |a|^2)h_2 =|\eps|^{1/2}\Re \Phi_2(\eps, \tmu,a,h),\\
(\varkappa\sgn(\eps) +  \Re\Lambda |h|^2+ \Re \Gamma |a|^2)h_2 +
(\chi\tmu + \Im\Lambda |h|^2+ \Im\Gamma |a|^2)h_1 =|\eps|^{1/2}\Im \Phi_2(\eps, \tmu,a,h),\\
\end{array}\right.\ee

\subsection{Case $a_1\neq 0$, $h=0$}
Denoting the left-hand side of \eqref{real} by $H_1$, we find easily that
at a root $(a_1,\tmu, h_1,h_2)=(a_*,\tmu_*, 0,0)$ of \eqref{real}:
\be\label{0jacobian}
J_1:=
\det
\frac{\partial H_1}{\partial (a_1,\tmu,b_1,b_2)}|_{(a_*,\tmu_*,0,0,0)}=
\begin{vmatrix}
2(\Re\Lambda) a_*^2 & 0 & 0 & 0\\
2(\Im\Lambda) a_*^2 & \chi a_* & 0 & 0\\
0 & 0 & \Re C_* & -\Im C_*\\
0 & 0 & \Im C_* & \Re C_*\\
\end{vmatrix}
=
2\chi(\Re \Lambda)(\Lambda -\Gamma)a_*^7,
\ee
where $C_*:=\varkappa+i\chi\tmu_* + \Gamma a_*^2= (\Lambda -\Gamma)a_*^2$ since
$ \varkappa + \rmi\chi\tmu_* + \Lambda a_*^2=0$.
 Thus, under genericity assumptions $\Re \Lambda \neq 0$,
$\Lambda\neq \Gamma$, and since $\chi\ne 0$, by Lemma~\ref{l4.5}, we have $J_1\neq 0$
and we can conclude by the Implicit Function Theorem the desired existence
and uniqueness of solutions of the $|\eps|^{1/2}$-perturbed system nearby those
of the $\eps=0$ one.
The case $h_1\neq 0$, $a=0$ goes symmetrically.

\subsection{Case $h=a\neq 0$}
This case is trickier due to the obvious nonuniqueness induced by rotational invariance of the truncated cubic order
system (specifically, the additional rotational invariance in common direction, or $S^1$-symmetry).
This will require a little bit different handling.
Specifically, note that we may by rotational symmetry {\it plus invariance under forward evolution/approximate rotation}
take {\it both} $a$ and $h$ to be real, and work with three unknowns
$(a_1,h_1,\tmu)$ and only three of the equation \eqref{real}.
This is enough to give uniqueness of solutions by Implicit Function Theorem, but not existence; existence on the other hand follows by reflective symmetry guaranteeing that $a=h$ is always a solution.

Denoting the first three lines on the left-hand side of \eqref{real} by $H_2$,
we have at a root $(a_1,\tmu, h_1)=(a_\natural,\tmu_\natural, a_\natural)$ of \eqref{real}:
\be\label{ejacobian}
J_2:=
\det
\frac{\partial H_2}{\partial (a_1,\tmu,h_1)}|_{(a_\natural,\tmu_\natural, a_\natural)}=
\begin{vmatrix}
2(\Re \Lambda) a_\natural^2 & 0 & 2(\Re \Gamma) a_\natural^2  \\
2(\Im \Lambda) a_\natural^2 & \chi a_\natural & 2(\Im \Gamma) a_\natural^2 \\
2(\Re \Gamma) a_\natural^2 & 0 & 2(\Re \Lambda) a_\natural^2  \\
\end{vmatrix}
=
4\chi(\Re \Lambda+\Re \Gamma)(\Re \Lambda -\Re \Gamma) a_\natural^5.
\ee
Alternatively, denoting the first, second, and fourth lines of the left-hand side of \eqref{real} as $H_3$, we have
\be\label{ejacobian'}
J_3:=
\det
\frac{\partial H_3}{\partial (a_1,\tmu,h_1)}|_{(a_\natural,\tmu_\natural, a_\natural)}=
\begin{vmatrix}
2(\Re \Lambda) a_\natural^2 & 0 & 2(\Re \Gamma) a_\natural^2  \\
2(\Im \Lambda) a_\natural^2 & \chi a_\natural & 2(\Im \Gamma) a_\natural^2 \\
2(\Im \Gamma) a_\natural^2 & \chi a_\natural & 2(\Im \Lambda) a_\natural^2  \\
\end{vmatrix}
=
4\chi(\Re \Lambda+ \Re \Gamma)(\Im \Lambda -\Im \Gamma)  a_\natural^5.
\ee

Under genericity conditions \eqref{hyp}, one of $J_2$ or $J_3$ does not vanish,
or else $\Re(\Lambda + \Gamma)$ and $(\Lambda-\Gamma)$ would both vanish,
a contradiction.
But, either of $J_2\neq 0$ or $J_3\neq 0$ is sufficient to give uniqueness
by an application of the Implicit Function Theorem.

\br
The reason the $h=0$ case works in standard fashion is that the solution is rotationally invariant, so rotation does not induce nonuniqueness.
Note that in the $h=0$ case we get both existence- which we need, since it does not follow from invariance considerations- and uniqueness, while
for $h=a$ case we get only uniqueness--which is all we need, since existence guaranteed by reflectional/rotational symmetry (both required for that).
\er

\subsection{Case $a=h=0$}
Finally, we must treat the trivial solution $(a,h)=(0,0)$.
Passing to real and imaginary parts in \eqref{original} and denoting the left-hand side of this system by $H_4$, we find that:
\be\label{00jacobian}
J_4:=
\det
\frac{\partial H_4}{\partial (a_1,a_2,h_1,h_2)}|_{(0,0,0,0)}=
\begin{vmatrix}
\varkappa & -\chi\tmu&0&0\\
\chi\tmu & \varkappa&0&0\\
0 & 0 & \varkappa & -\chi\tmu\\
0 & 0 & \chi\tmu & \varkappa\\
\end{vmatrix}
=
(\varkappa^2+\chi^2\mu^2)^2\neq 0
\ee
for arbitrary $\tmu$, hence by the Implicit Function Theorem
there is a unique solution $(a,b)(|\eps|^{1/2}, \tmu)$
for each $(\eps, \tmu)$ near $(0,\tmu_*)$.  But, this is already accounted
for by the trivial solution $(a,b)=(0,0)$.
Thus, we can conclude again in this case the desired existence
and uniqueness of solutions of the $\eps$-perturbed system nearby those
of the $\eps=0$ one.
This completes the argument and the paper.

\end{document}